\documentclass[bj]{imsart}

\usepackage{amsthm,amsmath,amssymb}
\usepackage{tikz}
\usepackage[numbers,comma,square,sort&compress]{natbib}
\usepackage{color}
\RequirePackage[colorlinks,urlcolor=my-blue,linkcolor=my-red,citecolor=my-green]{hyperref}
\definecolor{my-blue}{rgb}{0.0,0.0,0.8}
\definecolor{my-red}{rgb}{0.6,0.0,0.0}
\definecolor{my-green}{rgb}{0.0,0.5,0.0}
\definecolor{light-gray}{gray}{0.6} 
\definecolor{really-light-gray}{gray}{0.8}

\arxiv{arXiv:1905.13353}

\received{July 2019}
\revised{November 2019}

\usepackage[left=1in,top=1in,right=1in,bottom=1in]{geometry}

\startlocaldefs

\newtheorem{theorem}{Theorem}[section]
\newtheorem{proposition}[theorem]{Proposition}
\newtheorem{lemma}[theorem]{Lemma}

\theoremstyle{definition}
\newtheorem{definition}[theorem]{Definition}

\theoremstyle{remark}
\newtheorem{remark}[theorem]{Remark}

\numberwithin{equation}{section}

\newcommand{\R}{\mathbb{R}}
\newcommand{\N}{\mathbb{N}}

\newcommand{\E}{\mathbb{E}}

\newcommand{\mx}{\mathbf{x}}
\newcommand{\my}{\mathbf{y}}
\newcommand{\ms}{\mathbf{s}}
\newcommand{\mz}{\mathbf{z}}
\newcommand{\mB}{\mathbf{B}}

\newcommand{\Z}{\mathbb{Z}}

\newcommand{\tZ}{\widetilde{Z}}

\newcommand{\1}[1]{\mathbf{1} \left \{ #1 \right \}} 

\endlocaldefs

\begin{document}

\begin{frontmatter}

\title{\bf Busemann functions and semi-infinite  O'Connell-Yor polymers}

\runtitle{Busemann functions and semi-infinite O'Connell-Yor polymers}

\begin{aug}
\author{\fnms{Tom} \snm{Alberts}\thanksref{a}\corref{}\ead[label=e1]{alberts@math.utah.edu}},
\author{\fnms{Firas} \snm{Rassoul-Agha}\thanksref{a}\ead[label=e2]{firas@math.utah.edu}}
\and
\author{\fnms{Mackenzie} \snm{Simper}\thanksref{b}\ead[label=e3]{msimper@stanford.edu}}
\address[a]{Department of Mathematics, University of Utah,
155 S 1400 E Rm 233, SLC, UT, USA 84112
\printead{e1};\printead{e2}}
\address[b]{Department of Mathematics, Stanford University,
Building 380, Stanford, CA, USA 94305
\printead{e3};}

\runauthor{T. Alberts, F. Rassoul-Agha, M. Simper}

\affiliation{University of Utah and Stanford University}

\end{aug}

\begin{abstract}
We prove that given any fixed asymptotic velocity, the finite length O'Connell-Yor polymer has an infinite length limit satisfying the law of large numbers with this velocity. By a Markovian property of the quenched polymer this reduces to showing the existence of \textit{Busemann functions}: almost sure limits of ratios of random point-to-point partition functions. The key ingredients are the Burke property of the O'Connell-Yor polymer and a comparison lemma for the ratios of partition functions. We also show the existence of infinite length limits in the Brownian last passage percolation model.
\end{abstract}




\end{frontmatter}


\section{Introduction \label{sec:intro}}

Among the class of $(1+1)$-dimensional exactly solvable statistical mechanics models in the Kardar-Parisi-Zhang (KPZ) universality class, the O'Connell-Yor polymer has several notable features that make it particularly fertile for analysis. First, the random environment is made up of independent copies of Brownian motion and hence inherits all of its rich features. At the same time, the polymer paths are Poisson-like processes which only make jumps of size $1$ at a discrete set of random times, whose law is determined by the random environment. This discreteness makes analysis of the path measures much simpler than fully continuous objects such as the continuum directed random polymer \cite{AKQ:CDRP, AKQ:IDR}, while at the same time retaining all of the salient features.

In this paper we prove an existence result for O'Connell-Yor polymers of infinite length. The quenched measures on finite length paths are defined by the standard Gibbs measure formulation (see Definition \ref{defn:p2l}), but it is unclear how to directly define an infinite length version that is compatible with the finite length measures. The primary hurdle is the non-consistency of the quenched measures on paths of increasing length, which renders useless standard tools such as the Kolmogorov consistency theorem. Nonetheless, an almost sure limit of the finite length path measures can still be taken thanks to three main ideas: expressing the transition probabilities for the finite length paths as a ratio of random partition functions, proving a particular comparison lemma for ratios of deterministic partition functions (in our case this is an extension of the comparison lemma in \cite{SV:ALEA2010}), and then using the stationarity property of the O'Connell-Yor polymer (based on the  Matsumoto-Yor property \cite{MY:stationarity}) to obtain stationary ratios that bound the ratios of interest. Finally, a monotonicity property of the stationary ratios gives a ``squeeze'' type proof for the existence of the limiting non-stationary ratios.
These limits are taken for point-to-point paths with endpoint moving at a fixed asymptotic velocity, the velocity being the ratio of the spatial displacement to the temporal one.  Such limits of ratios of partition functions are commonly referred to as \textit{Busemann functions}, and our main result is that for each fixed asymptotic velocity there exists a family of random Busemann functions indexed by the points of $(\Z \times \R)^2$. These are measurable functions of the Brownian motions ``ahead'' of the indexing point (in the componentwise ordering on $\Z \times \R$) and satisfy the \textit{cocycle property}. We also identify the limiting distribution of the Busemann function between any two fixed space-time points and for a fixed velocity.

Our results on the existence of Busemann functions can also be interpreted in terms of existence and uniqueness of global stationary solutions and pull-back attractors of a semi-discrete approximation of the stochastic Burgers equation
\begin{align}\label{Burgers}
\partial_t u=\tfrac\nu2 \partial_{xx}u+u\partial_x u+\partial_x\dot W,
\end{align}
where $\dot W$ is standard space-time white noise and $\nu\ge0$ is the viscosity parameter.
See Remark \ref{rk:RDS}.

This paper is made up of three subsequent sections and a short appendix. In the next section we recall the basic features of the O'Connell-Yor polymer and state our main results. In Section \ref{sec:Busemann} we prove that the limits of ratios of partition functions exist, using a comparison lemma and the explicitly known stationary version of the O'Connell-Yor polymer. The article \cite{GRSY:log_gamma} used this method for the directed log-gamma polymer \cite{Sepp:log_gamma}, where the stationary model is also explicit. See also \cite{DH:FPP_Busemann,DH:bigeodesics,GRS:corner_growth_cocycles,GRS:corner_growth_geodesics,RJ:DLR_Busemann}, which use ideas from this method in discrete-time discrete-space models outside the exactly solvable class. Our results are somewhat more complicated than these works in that we have to prove an almost sure tightness of the quenched path measures, which is necessary in our case because of the continuous time parameter but comes for free in the fully discrete setting. 

The present work is the first proof of existence of Busemann functions in a non fully discrete setting that makes use of the stationary boundary model. The works \cite{BCK:Burgers,CP:LPP_Busemann,CP:Hammersley,Bak:kick_forcing,BL:thermo_limit} also prove existence of Busemann functions in various related models, but using a different method based on path-straightness estimates, pioneered by Newman and coauthors \cite{LN:GeoFPP,HN:Euclidean}. Other novel features of the present work include a new version of the comparison lemma for ratios of partition functions, and in contrast to \cite{GRSY:log_gamma} our proof has been streamlined to avoid using the complicated time reversal property of the stationary polymer. This allows us to work with one stationary boundary condition throughout, instead of having to constantly switch between ``south-west'' and ``north-east'' boundary conditions as in \cite{GRSY:log_gamma}. Within Sections \ref{sec:OCY} and \ref{sec:Busemann} our main result is Theorem \ref{thm:ratios_of_part_funcs}, which precisely states existence of Busemann functions and the existence of infinite length O'Connell-Yor polymers. Its proof is broken into many subsequent steps, with the main contributions being Theorem \ref{thm:almost_sure_check_comparison} which bounds limits of ratios of partition functions in the stationary regime, and Lemma \ref{lm:comparison} which gives deterministic bounds on ratios of partition functions in finite size boxes that hold for any underlying environment field. Section \ref{sec:LPP} is an extension of the results in Sections \ref{sec:OCY} and \ref{sec:Busemann} that proves the existence of the infinite length Brownian last passage percolation model. Finally, an appendix uses ideas from large deviations to prove a local shape theorem for the free energy of directed O'Connell-Yor polymers within a prescribed range of asymptotic velocities, which is an extension of the free energy result of \cite{MO:free_energy}. \bigskip

\noindent \textbf{Notation.} We let $\Z_+$ denote the non-negative integers $\{0,1,2,\ldots \}$ and $\Z$ for the full set of integers. We use $\N$ for the positive integers. Points $(n,t) \in \Z \times \R$ will be used to denote the space-time locations of various jump processes, and we use the componentwise ordering $(m,s) \leq (n,t)$ on such points. For integers $m \leq n$ placeholder variables for the jump times of the process from $m$ to $n$ are denoted by
\[
\ms_{m,n} = (s_m, s_{m+1}, \ldots, s_n) \in \R^{n-m+1}.
\]
For points $(m,s), (n,t)$ with $(m,s) \leq (n,t)$ we define the ordered subsets of times
\begin{align*}
&\Pi_{(m,s), (n,t)} = \{ \ms_{m,n} : s = s_m < s_{m+1} < \ldots < s_n = t \}\quad\text{and}\\
&\Pi_{(m,s)} = \{ \ms_{m,n} : s = s_m < \ldots < s_n \}.
\end{align*}
We use $\1{A}$ to denote the indicator function of an event $A$. For functions $f : \R \to \R$ we frequently use the notation $f(s,t) = f(t) - f(s)$ for increments, without assuming that $s \leq t$. Recall the gamma function $\Gamma(\alpha) = \int_0^\infty x^{\alpha-1} e^{-x} \, dx$ and the digamma function defined by $\Psi_0 = \Gamma'/\Gamma$. We will also make use of the trigamma function $\Psi_1 = \Psi_0'$. \newline

\section{The O'Connell-Yor Polymer and Main Results \label{sec:OCY}}

In this section we define the polymer model. The paths of the polymer are c\`adl\`ag processes taking values in $\Z$  and with jump size always equal to one; in this way they are similar to the counting process for a Poisson point process on $\R$. We denote the jump times by $\{ \tau_j \}_{j \in \Z}$, where $\tau_j$ represents the time at which the process jumps from site $j$ onto site $j+1$. Thus if $x : (-\infty, \infty) \to \Z$ is the path then $\tau_j$ is the time at which $x(\tau_j-) = j$ and $x(\tau_j) = j+1$. Such processes are fully determined by their jump times since they are constant in between. Consequently, it is enough to specify the joint law of the jump times to fully determine the law of the paths. For the O'Connell-Yor polymer this is done by introducing a collection of independent two-sided Brownian motions $\mB = ( B_i(t), t \in \R )_{i \in \Z}$ to act as a random environment that the jump process interacts with. One may assume throughout that $B_i(0) = 0$ but we will only be concerned with increments of the $B_i$, so such normalizations are typically irrelevant. Recall that we denote the increments by $B_i(s,t) = B_i(t) - B_i(s)$. With this notation in hand the point-to-point version of the partition function of the O'Connell-Yor polymer is defined as follows.

\begin{definition}\label{defn:p2l}
Let $(m,s), (n,t) \in \Z \times \R$ with $(m,s) < (n,t)$. The \emph{point-to-point partition function} is defined as
\begin{align*}
Z_{(m,s),(n,t)}(\mB) &= \int e^{ \sum_{k=m}^n B_k(s_{k-1}, s_{k}) } \1{s = s_{m-1} < s_m < \ldots < s_{n-1} < s_n = t} \, d \ms_{m,n-1} \\
&= \int e^{ \sum_{k=m}^n B_k(s_{k-1}, s_{k}) } \1{\ms_{m-1,n} \in \Pi_{(m-1,s),(n,t)}} \, d \ms_{m,n-1}.
\end{align*}
When the Brownian motions in use are clear from context, we may suppress the notation and simply write $Z_{(m,s),(n,t)}$ instead of $Z_{(m,s),(n,t)}(\mB)$. In the case $m=n$ with $s \leq t$ we write
\[
Z_{(m,s),(m,t)} = e^{B_m(s,t)}.
\]
A quick computation using the definitions shows that 
these partition functions satisfy the supermultiplicativity property:
\begin{align}\label{supermult}
Z_{(m,s),(\ell,r)}Z_{(\ell,r),(n,t)}\le Z_{(m,s),(n,t)}
\end{align}
for all $(m,s)\le(\ell,r)\le(n,t)$ in $\Z\times\R$.

Given the partition function, the quenched polymer measure on $\Pi_{(m-1,s),(n,t)}$ is
\[
Q_{(m,s),(n,t)}^{\mB}( \tau_m \in ds_m, \ldots, \tau_{n-1} \in ds_{n-1} ) = \frac{1}{Z_{(m,s),(n,t)}(\mB)} e^{ \sum_{k=m}^n B_k(s_{k-1}, s_{k}) } \, d \ms_{m,n-1}.
\]
Note that this law only depends on the Brownian motions $B_m, \ldots, B_n$. Further note that under this measure the variables $\tau_{m-1}$ and $\tau_n$ are fixed at $s$ and $t$, respectively, while $\tau_m, \ldots, \tau_{n-1}$ are random.  
\end{definition}

\begin{remark}
The definition of the partition functions implies that they satisfy a supermultiplicativity property: if $(m,s) \leq (k,r) \leq (n,t)$ then $Z_{(m,s), (k,r)} Z_{(k,r),(n,t)} \leq Z_{(m,s),(n,t)}$. We will also occasionally use that the expected value of the partition function is given by
\[
\E[Z_{(m,s),(n,t)}(\mB)] = e^{(t-s)/2} \frac{(t-s)^{n-m}}{(n-m)!},
\]
which follows from Fubini's theorem.
\end{remark}

\begin{remark}
One could introduce an inverse temperature parameter $\beta$ into the above definition, but the Brownian scaling property $B_k^{\beta}(\cdot) := \beta B_k(\beta^{-2} \cdot) \equiv B_k(\cdot)$ (here $\equiv$ means equality in law) gives 
\[
\beta^{2(n-m)} Z_{(m,s),(n,t)}(\beta \mB) = Z_{(m,\beta^2 s), (n, \beta^2 t)} (\mB^{\beta}) \equiv Z_{(m, \beta^2 s), (n, \beta^2 t)}(\mB).
\]
From this equality and Definition \ref{defn:p2l} it follows that 
\[
(\tau_m, \ldots, \tau_{n-1}) \sim Q_{(m,s), (n,t)}^{\beta \mB} \iff (\beta^2 \tau_m, \ldots, \beta^2 \tau_{n-1}) \sim Q_{(m, \beta^2 s), (n, \beta^2 t)}^{\mB^{\beta}}.
\]
Now since $\tau_k(X(\beta^{-2} \cdot)) = \beta^{2} \tau_k(X(\cdot))$, the latter implies that 
\[
X \sim Q_{(m,s), (n,t)}^{\beta \mB} \iff X(\beta^{-2} \cdot) \sim Q_{(m,\beta^2 s), (n, \beta^2 t)}^{\mB^{\beta}}. 
\]
For these reasons we choose to keep $\beta = 1$ throughout.


\end{remark}

Note that the quenched measure determines the process by fully specifying the joint law of all of the jump times $\{ \tau_k \}_{k=m}^{n-1}$. The next lemma shows that all marginal measures of this joint law can be expressed as ratios of partition functions.  

\begin{lemma}\label{lem:quenched_marginals}
Under \(Q^{\mB}_{(m,s), (n,t)} \), the marginals of \( \{\tau_k \}_{k = m}^{n-1} \)  are a product of point-to-point partition functions. More precisely, for integers $m \leq k_1 < k_2 < \ldots < k_{\ell} \leq n-1$ the law of $(\tau_{k_1}, \ldots, \tau_{k_{\ell}})$ under $Q^{\mB}_{(m,s),(n,t)}$ is
\begin{equation} \label{eqn: marginalDist} \frac{Z_{(m,s), (k_1, s_1)}(\mB) \cdot \prod_{i = 1}^{l-1} Z_{ (k_i + 1, s_i), (k_{i+1}, s_{i+1})}(\mB)  \cdot Z_{(k_{\ell} + 1, s_{\ell}), (n,t)}(\mB)}{Z_{(m,s), (n,t)}(\mB) } \, ds_{1} \ldots s_{\ell},
\end{equation}
on the chamber $\{ s < s_1 < \ldots < s_{\ell} < t \}$. Consequently, all conditional densities of one subset of the jump times given another subset can also be expressed as ratios of partition functions. In particular, under the quenched point-to-point measure the jump times $\tau_k$ form a Markov process with transition densities given by
\begin{equation} \label{eqn: condDist}
\begin{split}
&Q^{\mB}_{(m,s), (n,t)} \left( \tau_{k_i} \in ds_i \vert  \tau_{k_{i-1}} \in ds_{i-1} \right)\\ 
&\qquad\qquad= \frac{   Z_{(k_{i-1}+1, s_{i-1}), (k_i, s_i)}(\mB) \cdot Z_{(k_i + 1, s_i), (n,t)}(\mB)}{Z_{(k_{i-1}+1, s_{i-1}), (n,t)}(\mB) } \1{s_{i-1} < s_i} \, ds_i.
\end{split}
\end{equation}
\end{lemma}

The proof of Lemma \ref{lem:quenched_marginals} is a straightforward exercise using the definition of the partition functions. An equivalent statement is that under $Q_{(m,s),(n,t)}^{\mB}$ the quenched path process $X$ is a continuous-time Markov chain that starts at $X(s) = m$ and makes jumps of size one, with inhomogeneous space-time rate
\[
\frac{Z_{(X(u)+1, u), (n,t)}(\mB)}{Z_{(X(u), u),(n,t)}(\mB)},
\]
until it hits level $n$. Using these descriptions, the problem of showing that limits of quenched path measures exist for the O'Connell-Yor polymer reduces to proving that certain ratios of partition functions have limits as $n \to \infty$. For integers $m \leq k_1 < k_2 < \ldots < k_{\ell}$, we want to show that the law of
$(\tau_{k_1}, \ldots, \tau_{k_\ell})$ under $Q_{(m,s), (n, n \theta)}^{\mB}$ converges as $n \to \infty$, and does so almost surely with respect to the Brownian motions $\mB$. By \eqref{eqn: marginalDist}, it is sufficient to show existence of the almost sure limit
\[
\lim_{n \to \infty} \frac{Z_{(k_{\ell}+1, s_{\ell}), (n,n \theta)}(\mB)}{Z_{(m,s), (n,n \theta)}(\mB)}.
\]
Our main theorem proves exactly this.

\begin{theorem}\label{thm:ratios_of_part_funcs}
Fix $\theta > 0$ and $\mx,\my\in\Z\times\R$. Then with probability one there exists a limit of the ratio of point-to-point partition functions starting from $\mx$ and $\my$, i.e.
\begin{align}\label{eqn:Busemann_exist_statement}
\lim_{n \to \infty} \frac{Z_{\mx,(n, t_n)}(\mB)}{Z_{\my,(n, t_n)}(\mB)} =: e^{\mathcal{B}^{\theta}(\mx, \my)}
\end{align}
exists almost surely and is independent of the choice of the sequence  $\{t_n \}$ in $\R$, so long as $t_n/n\to\theta$. Furthermore, for each fixed $(m,s)$ the limit of $Q_{(m,s), (n,t_n)}^{\mathbf{B}}$ {\rm(}in the sense of weak convergence of measures on the Skorohod space $\mathcal{D}[m, \infty)${\rm)} exists almost surely as $n \to \infty$. Under the limiting measure the quenched path process $X$ is a continuous-time Markov chain that starts at $X(s) = m$ and makes jumps of size one, with inhomogeneous space-time rate 
\[
e^{-\mathcal{B}^{\theta}((X(u), u), (X(u)+1,u))}
\]
at time $u$. 
\end{theorem}


The limit $\mathcal{B}^{\theta}(\mx, \my)$ is referred to as the Busemann function between $\mx$ and $\my$, corresponding to the velocity $\theta$. By construction it satisfies the cocycle property
\[
\mathcal{B}^{\theta}(\mx, \my) = \mathcal{B}^{\theta}(\mx, \mz) + \mathcal{B}^{\theta}(\mz, \my) 
\]
for any $\mz \in \Z \times \R$. 


\begin{remark}
Note that we do not require that $\mx, \my$ are ordered to prove the existence of the limit. However, as we will now demonstrate, it is sufficient to show existence of the almost sure limits of the type
\[
\lim_{n \to \infty} \frac{Z_{(1,0), (n,n \theta)}(\mB)}{Z_{(0,0), (n,n \theta)}(\mB)} \quad \textrm{and} \quad \lim_{n \to \infty} \frac{Z_{(0,t), (n,n \theta)}(\mB)}{Z_{(0,0), (n, n \theta)}(\mB)}.
\]
First, for any choice of $\mx$, $\my$ at least one of the points can always be translated back to the origin by translating the field of Brownian motions in the same way. Furthermore, by introducing extra terms any ratio can be expressed as a product of a sequence of ratios in which all ``starting points'' for the $Z$ are vertically or horizontally aligned. Second, the specific choice of $t_n = n \theta$ is sufficient because once it is proved for constant velocities a comparison principle and a squeeze type theorem can be used to show that it holds for all arbitrary sequences with a fixed asymptotic velocity; this is done in Section \ref{sec:subsq}. Finally, proving almost sure convergence of the path measures $Q_{(m,s),(n,t)}^{\mB}$ requires showing that adjacent jumps do not, with positive probability, merge into a single jump of size larger than one. This is done in Section \ref{sec:tightness}, using standard modulus-of-continuity estimates for Brownian motion and a first moment formula for the random partition function.
\end{remark}


\begin{remark}
The existence results of Theorem \ref{thm:ratios_of_part_funcs} can also be cast in a statistical mechanics framework, where the semi-infinite quenched path measures we construct via Busemann functions correspond to semi-infinite Gibbs measures that are consistent with the finite path point-to-point polymer measures $Q_{(m,s),(n,t)}^{\mB}$. See \cite{RJ:DLR_Busemann} for details in the polymer model on $\Z^2$. 
\end{remark}

\begin{remark}
\label{rk:RDS}
A direct differentiation shows that for any $m \in \Z$ the partition functions $Z_{(m,0),(n,t)}(\mB)$  solve the infinite system of coupled stochastic differential equations (SDEs)
    \begin{align}\label{semi-disc}
    dZ_{(m,0),(n,t)}(\mB)=Z_{(m,0),(n-1,t)}(\mB) \, dt + Z_{(m,0),(n,t)}(\mB) \, dB_n(t), \quad n > m,\ t\in\R.
    \end{align}
By \cite{MN:Intermediate} we know that after an appropriate scaling,  $\{Z_{(0,0),(n,t)}:n\in\Z,t\in\R\}$ converges weakly to the solution of the stochastic heat equation (SHE)
    \[\partial_t\mathcal Z=\tfrac\nu2\partial_{xx}\mathcal Z+\tfrac1\nu\mathcal Z\dot W\]
with $\delta_0$ initial condition. Then, 
$\nu\partial_x\log\mathcal Z$ is the Hopf-Cole solution of the stochastic viscous Burgers equation \eqref{Burgers}. Thus, equations \eqref{semi-disc} can be thought of as a semi-discrete approximation of the stochastic Burgers equation \eqref{Burgers}. 
The results of Theorem \ref{thm:ratios_of_part_funcs}
can  be  reinterpreted as results on the existence and uniqueness of
global skew-invariant solutions and pull-back attractors, as $m \to -\infty$, of the random dynamical system given by the above system of coupled SDEs. For a similar situation, see for example \cite[Theorem 3.2]{Bak:kick_forcing}, \cite[Theorem 3.8]{RJ:DLR_Busemann}, and the forthcoming \cite{RJ:1F1S}. 
\end{remark}

\subsection{The Burke Property of the O'Connell-Yor Polymer} \label{sec: stationaryModel}

Here we recall the results of \cite{OCY:SPA2001} and use them to define the stationary version of our polymer model. The setup we use is borrowed from \cite{SV:ALEA2010}, which gives a fuller description of the stationary situation. The stationary model is created by using the initial Brownian motion $B_0$ in a particular way. One may think of its increments as being the increments of the Busemann functions for polymers coming from infinitely far in the past, but obviously we cannot define it this way without first knowing that Busemann functions exist. Instead we use its exponential as the initial condition for the stochastic heat equation with multiplicative noise, which we define next and from it define the stationary model. 

\begin{definition}\label{defn:stationary_Z}
For \(N \in \Z_+ \), \(T \in \R \), and $\lambda > 0$ define the field  \(\tZ^\lambda_{(N, T)} \) by the recursion relation
\begin{align*}\label{defn:stationary_Z}
\tZ^\lambda_{(N, T)}  = \int_{- \infty}^T e^{B_N(u, T)} \tZ^\lambda_{(N-1, u)} \lambda e^{-\lambda(T - u)} \, du,
\end{align*}
for \(N > 0 \), and for $N=0$ use the initial condition
\[
\tZ^{\lambda}_{(0,T)} = e^{-B_0(T)}. 
\]
\end{definition}

Note this recursion means that $\tZ^{\lambda}$ is the Feynman-Kac solution to the stochastic heat equation for a Poisson process interacting with the Brownian motions, in the form of a multiplicative noise.

The ratios of these partition functions play an important role, so we set special notation for the ratios in both the space and time directions.

\begin{definition}\label{defn:r_g_check}
Fix a $\lambda > 0$. For $N \in \N$ and $T \in \R$ set
\begin{align}\label{eqn:r_g_check_1}
r_N^{\lambda}(T) = \log \tZ^\lambda_{(N,T)} - \log \tZ^\lambda_{(N-1,T)}\quad\text{and} \quad g_N^{\lambda}(T) = - \log \tZ^\lambda_{(N,T)}. 
\end{align}
Note that by taking differences of the function $r_N^{\lambda}$ we are led to the recursion relation
\begin{align}\label{eqn:g_recursion}
g_N^{\lambda}(S,T) = g_{N-1}^{\lambda}(S,T) - r_N^{\lambda}(S,T), \quad N \in \N.
\end{align}
Here $g_0^{\lambda}(S,T) = B_0(S,T)$, which is consistent with \eqref{eqn:r_g_check_1} and the definition of $\tilde{Z}^{\lambda}_{(0,T)}$. Finally, we define a new sequence of fields $\check{B}^{\lambda}$ by
\[
\check{B}^{\lambda}_{N-1}(S,T) = B_N(S,T) - r_N^{\lambda}(S,T), \quad N \in \N.
\]
\end{definition}

The various fields involved here satisfy an important stationarity property: for any down-right path the fields $B$, $r^{\lambda}$, $g^{\lambda}$ and $\check{B}^{\lambda}$ are all independent within certain regions. The precise statement is given below, see Figure \ref{fig: burkeProperty} for an illustrative statement. 


\begin{figure}[h!]

\begin{tikzpicture}[>=latex]
%


\draw[line width=1pt](5,0)--(6.8,0);
\draw(6,0)node[above]{\scriptsize$B_0$};

\draw[line width=1pt, color=my-blue](-5.9,0)--(5,0);
\draw(2.5,0)node[color=my-blue,above]{\scriptsize${\check B}^\lambda_0$};

\draw[line width=1pt, color=my-red](5,0)--(5,0.5);
\draw(5,0.25)node[color=my-red,left]{\scriptsize$r^\lambda_1$};

\draw[line width=1pt, color=my-green](3.5,0.5)--(5,0.5);
\draw(4.2,0.5)node[color=my-green,above]{\scriptsize$g^\lambda_1$};

\draw[line width=1pt, color=my-blue](-5.9,0.5)--(3.5,0.5);
\draw(1.5,0.5)node[color=my-blue,above]{\scriptsize${\check B}^\lambda_1$};

\draw[line width=1pt](5,0.5)--(6.8,0.5);
\draw(6,0.5)node[above]{\scriptsize$B_1$};

\draw[line width=1pt, color=my-red](3.5,0.5)--(3.5,1);
\draw(3.5,0.75)node[color=my-red,left]{\scriptsize$r^\lambda_2$};

\draw[line width=1pt, color=my-green](2,1)--(3.5,1);
\draw(2.7,1)node[color=my-green,above]{\scriptsize$g^\lambda_2$};

\draw[line width=1pt, color=my-blue](-5.9,1)--(2,1);
\draw(0.2,1)node[color=my-blue,above]{\scriptsize${\check B}^\lambda_2$};

\draw[line width=1pt](3.5,1)--(6.8,1);
\draw(4,1)node[above]{\scriptsize$B_2$};

\draw[line width=1pt, color=my-red](2,1)--(2,1.5);
\draw(2,1.25)node[color=my-red,left]{\scriptsize$r^\lambda_3$};

\draw[line width=1pt, color=my-green](1.3,1.5)--(2,1.5);
\draw[line width=1pt](2,1.5)--(6.8,1.5);
\draw(2.5,1.5)node[above]{\scriptsize$B_3$};

\draw(-0.1,1.9)node[color=my-green]{$\ddots$};
\draw(5,2.1)node{$\vdots$};
\draw(-2.5,1.6)node[color=my-blue]{$\vdots$};

\draw[line width=1pt, color=my-red](-1.5,2)--(-1.5,2.5);
\draw(-1.5,2.25)node[color=my-red,left]{\scriptsize$r^\lambda_{n-2}$};

\draw[line width=1pt, color=my-blue](-5.9,2)--(-1.5,2);
\draw(-3.5,2)node[color=my-blue,above]{\scriptsize${\check B}^\lambda_{n-3}$};
\draw[line width=1pt, color=my-green](-1.5,2)--(-.8,2);

\draw[line width=1pt, color=my-green](-3,2.5)--(-1.5,2.5);
\draw(-2.3,2.5)node[color=my-green,above]{\scriptsize$g^\lambda_{n-2}$};

\draw[line width=1pt, color=my-blue](-5.9,2.5)--(-3,2.5);
\draw(-4.5,2.5)node[color=my-blue,above]{\scriptsize${\check B}^\lambda_{n-2}$};

\draw[line width=1pt](-1.5,2.5)--(6.8,2.5);
\draw(-.5,2.5)node[above]{\scriptsize$B_{n-2}$};

\draw[line width=1pt, color=my-red](-3,2.5)--(-3,3);
\draw(-3,2.75)node[color=my-red,left]{\scriptsize$r^\lambda_{n-1}$};

\draw[line width=1pt, color=my-green](-4.3,3)--(-2.8,3);
\draw(-3.7,3)node[color=my-green,above]{\scriptsize$g^\lambda_{n-1}$};

\draw[line width=1pt, color=my-blue](-5.9,3)--(-4.3,3);
\draw(-5.3,3)node[color=my-blue,above]{\scriptsize${\check B}^\lambda_{n-1}$};

\draw[line width=1pt](-3,3)--(6.8,3);
\draw(-2.3,3)node[above]{\scriptsize$B_{n-1}$};

\draw[line width=1pt, color=my-red](-4.3,3)--(-4.3,3.5);
\draw(-4.3,3.25)node[color=my-red,left]{\scriptsize$r^\lambda_n$};

\draw[line width=1pt, color=my-green](-5.9,3.5)--(-4.3,3.5);
\draw(-5,3.5)node[color=my-green,above]{\scriptsize$g^\lambda_n$};

\draw[line width=1pt](-4.3,3.5)--(6.8,3.5);
\draw(-3.5,3.5)node[above]{\scriptsize$B_n$};

\draw[fill=black](5,0)circle(0.5mm)node[below]{\scriptsize$t_1$};
\draw[fill=black](3.5,0)circle(0.5mm)node[below]{\scriptsize$t_2$};
\draw[fill=black](2,0)circle(0.5mm)node[below]{\scriptsize$t_3$};

\draw(0.2,0)node[below]{$\cdots$};

\draw[fill=black](-1.5,0)circle(0.5mm)node[below]{\scriptsize$t_{n-2}$};
\draw[fill=black](-3,0)circle(0.5mm)node[below]{\scriptsize$t_{n-1}$};
\draw[fill=black](-4.3,0)circle(0.5mm)node[below]{\scriptsize$t_n$};

\end{tikzpicture}

\caption{Independence structure of the $\check{B}, r, g$, and $B$ fields.}
\label{fig: burkeProperty}
\end{figure}
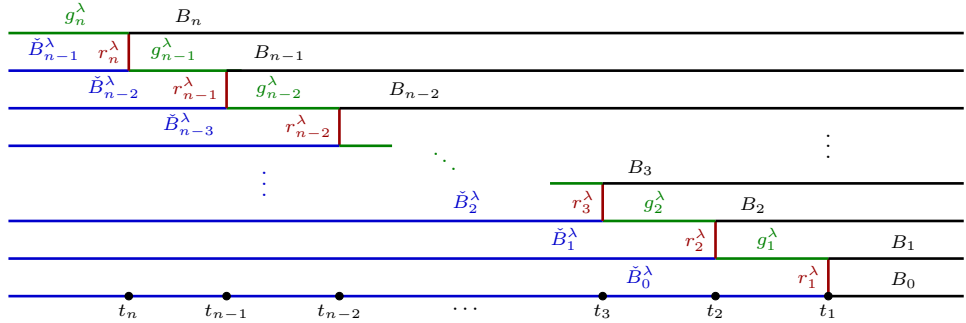

\begin{theorem}[Burke Property of the O'Connell-Yor Stationary Polymer, \cite{OCY:SPA2001,SV:ALEA2010}]\label{thm:burke}
For $n \in\N$ consider times $-\infty < t_n \leq t_{n-1} \leq \ldots \leq t_1 < \infty$. Then the processes
\begin{align*}
&\{ g_n^{\lambda}(s, t_n) : s \leq t_n \}, \ \{ B_n(t_n, s) : s \geq t_n \}, \  \{ \check{B}_0^{\lambda}(s, t_1) : s \leq t_1 \}, \ \{ B_0(t_1, s) : s \geq t_1 \},\\
&\{ g_j^{\lambda}(t_{j+1}, s) : t_{j+1} \leq s \leq t_j \}, \ \{ \check{B}_j^{\lambda}(s,t_{j+1}) : s \leq t_{j+1} \}, \ \{ B_j(t_j,s) : s \geq t_j \},\\ 
&1 \leq j \leq n-1,
\end{align*}
are mutually independent Brownian motions, and they are independent of the random variables $r_i^{\lambda}(t_i), 1 \leq i \leq n$, which are also iid and have distribution $e^{-r_i^{\lambda}(t_i)} \sim \operatorname{Gamma}(\lambda, 1)$.
\end{theorem}

We take note of several important identities in these definitions that will be useful later on. First, from the definition of $r_N^{\lambda}$  we clearly have
\[
\sum_{k=1}^{N} r_k^{\lambda}(T) = \log \tZ^{\lambda}_{(N,T)} + B_0(T).
\]
In addition, using Definition \ref{defn:stationary_Z} and the definition of $g_N^{\lambda}$, we can construct \(r_N^\lambda \)  from the underlying Brownian motion \(B_{N} \) and the \(g_{N-1}^\lambda \) variables:
\[
r_N^\lambda (T) = \log \int_{-\infty}^T e^{B_{N}(s, T) + g_{N-1}^\lambda(s, T) - \lambda(T - s)} \, ds.
\]
From \cite[Lemma 3.2]{SV:ALEA2010} (which they attribute to \cite{OCY:SPA2001}), we also have the following involution-type identity for constructing $r_N^{\lambda}$ from $\check{B}^{\lambda}_{N-1}$ and $g_N^{\lambda}$.
\begin{align} \label{lem: involutionIdentity}
r_N^\lambda (T) = \log \int_T^\infty e^{\check{B}^\lambda_{N-1}(T, s) + g_{N}^\lambda(T, s) + \lambda(T-s)} \, ds.
\end{align}

\subsection{Ratios of Stationary Partition Functions}

For our proof, the crucial result needed is that ratios of appropriately defined partition functions using the weights $\check{\mB}^{\lambda}$ are independent of the height $N$. The new partition functions are of point-to-line type, meaning the last variable (typically $s_n$) is free in the integral that defines it, rather than fixed. They also pick up an extra weight on the terminal line, from a different collection of fields. Since we will use these types of partition functions repeatedly we introduce new notation for them.

\begin{definition}\label{defn:Z_bar}
Let $\mB = \{ B_i\}_{i \in \Z}$ and $\bar{\mB} = \{ \bar{B}_i \}_{i \in \Z}$ be independent fields of iid two-sided Brownian motions. For $\lambda > 0$, $(m,s) \in \Z \times \R$ and $n \geq m$ define the partition functions $\bar{Z}^{\lambda}_{(m,s),n}$ by
\[
\bar{Z}^{\lambda}_{(m,s),n}(\mB, \bar{\mB}) = \int_{\Pi_{(m,s),n}} \exp \left \{ \sum_{k=m}^n B_k(s_{k-1}, s_k) - \bar{B}_{n+1}(s_n) - \lambda s_n \right \} \, d \ms_{m,n}, 
\]
where we recall $\Pi_{(m,s),n} = \{ \ms_{m-1,n} : s = s_{m-1} < s_m < \ldots < s_n \}$.
\end{definition}

Note that the $\bar{Z}^{\lambda}$ partition function can be rewritten as
\begin{align}\label{eqn:Z_bar_alt}
\bar{Z}^{\lambda}_{(m,s),n}(\mB, \bar{\mB}) = \int_s^{\infty} e^{-\bar{B}_{n+1}(x) - \lambda x} Z_{(m,s),(n,x)}(\mB) \, dx.
\end{align}
We will use this identity repeatedly. The next lemma says that inputting the fields $\check{\mB}^{\lambda}$ and $g^{\lambda}$ into this partition function recovers the fields $B$ and $r^{\lambda}$.

\begin{lemma}\label{lem:stationary_ratios}
For $N \in\N$ the following identities hold:
\[
\frac{\bar{Z}_{(0,t), N-1}^{\lambda}(\check{\mB}^{\lambda}, -g^{\lambda})}{\bar{Z}_{(0,s), N-1}^{\lambda}(\check{\mB}^{\lambda}, -g^{\lambda})} = e^{B_0(s,t) - \lambda (t-s)} , \quad \frac{\bar{Z}_{(0,t), N}^{\lambda}(\check{\mB}^{\lambda}, -g^{\lambda})}{\bar{Z}_{(1,t), N}^{\lambda}(\check{\mB}^{\lambda}, -g^{\lambda})} =  e^{r_1^{\lambda}(t)}.
\]
In particular, the ratios are independent of the choice of $N$.
\end{lemma}

\begin{proof} 
For shorthand write simply $\bar{Z}^{\lambda}_{(m,s), n}(\check{\mB}^{\lambda}, -g^{\lambda}) = \bar{Z}^{\lambda}_{(m,s),n}$ and $Z_{(m,s),(n,x)}(\check{\mB}^{\lambda}) = Z_{(m,s),(n,x)}$. By using \eqref{eqn:Z_bar_alt} and making repeated use of \eqref{eqn:g_recursion} and \eqref{lem: involutionIdentity} we have
\begin{align*}
&\bar{Z}_{(m,s), n}^{\lambda} = \int_s^{\infty} e^{g_{n+1}^{\lambda}(x) - \lambda x} Z_{(m,s), (n, x)} \, dx \\
&= \int_s^{\infty} e^{g_{n+1}^{\lambda}(x) - \lambda x} \int_s^x e^{\check{B}_{n}^{\lambda}(s_{n-1},x)} Z_{(m,s), (n-1,s_{n-1})} \, ds_{n-1} \, dx \\
&= \int_s^{\infty}\!\!\! e^{g_{n+1}^{\lambda}(s_{n-1}) - \lambda s_{n-1}} Z_{(m,s), (n-1, s_{n-1})} \int_{s_{n-1}}^{\infty} \!\!\!\!e^{g_{n+1}^{\lambda}(s_{n-1},x) + \check{B}_n^{\lambda}(s_{n-1}, x) + \lambda(s_{n-1} - x)} \, dx \, ds_{n-1} \\
&= \int_s^{\infty} e^{g_{n+1}^{\lambda}(s_{n-1}) - \lambda s_{n-1}} Z_{(m,s), (n-1, s_{n-1})} e^{r_{n+1}^{\lambda}(s_{n-1})} \, d s_{n-1} \\
&= \int_s^{\infty} e^{g_n^{\lambda}(s_{n-1}) - \lambda s_{n-1}} Z_{(m,s), (n-1, s_{n-1})} \, d s_{n-1} = \bar{Z}_{(m,s),n-1}^{\lambda}.
\end{align*}
Thus we see the independence from $n$ immediately. Iterating gives
\begin{align*}
\bar{Z}_{(m,s),n}^{\lambda} &= \bar{Z}_{(m,s),m}^{\lambda} = \int_s^{\infty} e^{\check{B}_m^{\lambda}(s, s_m) + g_{m+1}^{\lambda}(s_m) - \lambda s_m} \, ds_m \\
&= e^{g_{m+1}^{\lambda}(s) - \lambda s + r_{m+1}^{\lambda}(s)} = e^{g_m^{\lambda}(s) - \lambda s},
\end{align*}
the last equality again following by \eqref{eqn:g_recursion}. The first part of the claim is a simple consequence of this last identity and $g_0^{\lambda}(s) = B_0^{\lambda}(s)$, while the second follows from the above and \eqref{eqn:g_recursion}:
\[
\frac{\bar{Z}_{(0,t),N}(\check{\mB}^{\lambda}, -g^{\lambda})}{\bar{Z}_{(1,t),N}(\check{\mB}^{\lambda}, -g^{\lambda})} = e^{g_0^{\lambda}(t) - g_1^{\lambda}(t)} = e^{r_1^{\lambda}(t)}. \qedhere
\]
\end{proof}

\section{Busemann Functions and the Existence of Infinite Length Limits \label{sec:Busemann}}

Given Brownian weights $\mB$ and $\lambda > 0$, construct $\check{\mB}^\lambda$ weights as in Section \ref{sec: stationaryModel}. The next theorem describes limits of ratios of partition functions in the dual check Brownian motions, as opposed to the original Brownian motions. We recall that $\Psi_1$ is the trigamma function, see the notation description at the end of Section \ref{sec:intro}.

\begin{theorem}\label{thm:almost_sure_check_comparison}
Fix real numbers $ t > 0, \lambda > 0$ and $\gamma > \Psi_1(\lambda) > \delta > 0$. Then with probability one
\begin{align}\label{eqn:almost_sure_check_comp_1}
\limsup_{N \to \infty} \frac{Z_{(0,0),(N, N\delta)}(\check{\mB}^{\lambda})}{Z_{(1, 0),(N, N \delta)}(\check{\mB}^{\lambda})} \leq e^{r_1^\lambda(0)} \leq \liminf_{N \to \infty} \frac{Z_{(0, 0), (N, N \gamma)}(\check{\mB}^{\lambda})}{Z_{(1, 0),(N, N \gamma)}(\check{\mB}^{\lambda})}.
\end{align}
Similarly, with probability one
\begin{align}\label{eqn:almost_sure_check_comp_2}
\limsup_{N \to \infty} \frac{Z_{(0,t),(N, N\delta)}(\check{\mB}^{\lambda})}{Z_{(0, 0),(N, N \delta)}(\check{\mB}^{\lambda})} \leq e^{B_0(t) - \lambda t} \leq \liminf_{N \to \infty} \frac{Z_{(0, t), (N, N \gamma)}(\check{\mB}^{\lambda}), }{Z_{(0,0),(N, N \gamma)}(\check{\mB}^{\lambda})}.
\end{align}
\end{theorem}
Note that the inequalities in Theorem \ref{thm:almost_sure_check_comparison} are both almost sure and come from the explicit construction of the $\check{\mB}^{\lambda}$ weights. Our main theorems are for the partition functions coming from the \textit{original} Brownian motions $\mB$, for which we have no such almost sure inequalities. However, we can turn the inequalities above into distributional inequalities for the partition functions derived from weights $\mB$, and then these distributional inequalities allow us to prove Theorem \ref{thm:ratios_of_part_funcs} in the case $t_n = n \theta$. We next give this proof of Theorem \ref{thm:ratios_of_part_funcs}, assuming Theorem \ref{thm:almost_sure_check_comparison}.

\begin{proof}[Proof of Theorem \ref{thm:ratios_of_part_funcs}, assuming Theorem \ref{thm:almost_sure_check_comparison}]
First, recall that $\check{\mB}^{\lambda} \equiv \mB$ for each $\lambda > 0$. Therefore, letting $\lesssim$ denote stochastic domination and inserting $\mB$ in place of $\check{\mB}^{\lambda}$ into the ratios in \eqref{eqn:almost_sure_check_comp_1}, this implies that for any $\lambda, \gamma, \delta > 0$ satisfying the conditions of Theorem \ref{thm:almost_sure_check_comparison} 
\begin{align}\label{eqn:limit_comparision}
\limsup_{N \to \infty} \frac{Z_{(0,0),(N, N \delta)}(\mB)}{Z_{(1, 0),(N, N \delta)}(\mB)}  \lesssim e^{r_1^{\lambda}(0)} \lesssim \liminf_{N \to \infty} \frac{Z_{(0,0),(N, N \gamma)}(\mB)}{Z_{(1, 0),(N, N \gamma )}(\mB)}.
\end{align}
Now fix $\theta > 0$. The left hand inequality of \eqref{eqn:limit_comparision} gives
\[
\limsup_{N \to \infty} \frac{Z_{(0,0),(N, N \theta)}(\mB)}{Z_{(1, 0),(N, N \theta)}(\mB)} \lesssim e^{r_1^{\lambda}(0)}
\]
for all $\lambda$ such that $\theta < \Psi_1(\lambda)$. Since the trigamma function is monotonically decreasing, this is equivalent to $\lambda \in (0, \Psi_1^{-1}(\theta))$. Now recall Theorem \ref{thm:burke} says $e^{-r_1^{\lambda}(0)} \sim \operatorname{Gamma}(\lambda,1)$, and since the $\operatorname{Gamma}(\lambda,1)$ family is stochastically increasing in $\lambda$, it follows that $e^{r_1^{\lambda}(0)}$ is stochastically decreasing in $\lambda$. Thus by letting $\lambda_{\theta} = \Psi_1^{-1}(\theta)$ and taking a limit as $\lambda \nearrow \lambda_{\theta}$, it follows that
\[
\limsup_{N \to \infty} \frac{Z_{(0,0),(N, N \theta)}(\mB)}{Z_{(1, 0),(N, N \theta)}(\mB)} \lesssim e^{r_1^{\lambda_{\theta}}(0)}.
\]
Similarly, the right hand side of \eqref{eqn:limit_comparision} gives
\[
e^{r_1^{\lambda}(0)} \lesssim \liminf_{N \to \infty} \frac{Z_{(0,0),(N, N \theta)}(\mB)}{Z_{(1, 0),(N, N \theta)}(\mB)}
\]
for all $\lambda \in (\Psi^{-1}(\theta), \infty) = (\lambda_{\theta}, \infty)$. Taking $\lambda \searrow \lambda_{\theta}$ and using the stochastic domination above gives
\[
\limsup_{N \to \infty} \frac{Z_{(0,0),(N, N \theta)}(\mB)}{Z_{(1, 0),(N, N \theta)}(\mB)} \lesssim e^{r_1^{\lambda_{\theta}}(0)} \lesssim \liminf_{N \to \infty} \frac{Z_{(0,0),(N, N \theta)}(\mB)}{Z_{(1, 0),(N, N \theta)}(\mB)}.
\]
But the liminf can stochastically dominate the limsup if and only if the two are in fact equal, and therefore the limit exists. Note that this argument also identifies the distribution of the limit of the ratios. The argument for the horizontal ratios of Theorem \ref{thm:ratios_of_part_funcs} is identical.
\end{proof}

\begin{remark}
The same argument identifies the limits for ratios at finally many space-time points. Combined with the cocycle property 
this identifies the finite-dimensional marginals of the process $(\mx, \my) \mapsto \mathcal{B}^{\theta}(\mx, \my)$.
\end{remark}

To prove Theorem \ref{thm:almost_sure_check_comparison} we require an intermediate comparison lemma. The next section sets about proving this lemma. They will also be used to extend the proof of Theorem \ref{thm:ratios_of_part_funcs} to the case of arbitrary sequences $t_n$ with a prescribed asymptotic slope.

\subsection{Comparison Lemma}

The comparison lemma considers the partition functions $\bar{Z}^{\lambda}$ restricted to certain events. For $\lambda > 0$ and $(m,s) \leq (n,t)$ we define
\[
\bar{Z}_{(m,s),n}^{\lambda}(\mB, \bar{\mB}; \tau_n < t) = \int_{\Pi_{(m,s),n}} \!\!\!\!\!\! e^{ -\lambda s_n - \bar{B}_{n+1}(s_n) + \sum_{k=m}^{n} B_k(s_{k-1}, s_k) } \1{s_n < t} \, d \ms_{m,n} 
\]
and an analogous partition function with the inequality switched in the indicator function, i.e.
\[
\bar{Z}_{(m,s),n}^{\lambda}(\mB, \bar{\mB}; \tau_n > t) = \int_{\Pi_{(m,s),n}} \!\!\!\!\!\! e^{ -\lambda s_n - \bar{B}_{n+1}(s_n) + \sum_{k=m}^{n} B_k(s_{k-1}, s_k) } \1{ s_n > t } \, d \ms_{m,n}.
\]
Note that these partition functions use the Brownian motions on both of the terminal lines at heights $m$ and $n+1$, respectively. They are not point-to-point polymers as the variable $s_{n}$ is free. However we do have the following comparison lemma with the point-to-point versions of Definition \ref{defn:p2l}.

\begin{lemma}\label{lm:comparison}
Fix $\lambda > 0$. Let $t > 0$ and $n \in\N$. Then
\begin{align}\label{eqn:comparison1}
\frac{\bar{Z}^{\lambda}_{(0,0),n}(\mB, \bar{\mB}; \tau_n < t)}{\bar{Z}^{\lambda}_{(1,0),n}(\mB, \bar{\mB}; \tau_n < t)} \leq \frac{Z_{(0,0),(n,t)}(\mB)}{Z_{(1,0),(n,t)}(\mB)} \leq \frac{\bar{Z}^{\lambda}_{(0,0),n}(\mB, \bar{\mB}; \tau_n > t)}{\bar{Z}^{\lambda}_{(1,0),n}(\mB, \bar{\mB}; \tau_n > t)}.
\end{align}
Similarly, let $0 < s < t < T$ and $n \in\Z_+$. Then 
\begin{align}\label{eqn:comparison2}
\frac{\bar{Z}^{\lambda}_{(0,t),n}(\mB, \bar{\mB}; \tau_n < T)}{\bar{Z}^{\lambda}_{(0,s),n}(\mB, \bar{\mB}; \tau_n < T)} \leq \frac{Z_{(0,t),(n,T)}(\mB)}{Z_{(0,s),(n,T)}(\mB)} \leq \frac{\bar{Z}^{\lambda}_{(0,t),n}(\mB, \bar{\mB}; \tau_n > T)}{\bar{Z}^{\lambda}_{(0,s),n}(\mB, \bar{\mB}; \tau_n > T)}.
\end{align}
\end{lemma}

\begin{remark}
As we will see from the proof, Lemma \ref{lm:comparison} is a deterministic statement that does not rely upon $\mB$ or $\bar{\mB}$ being a field of iid Brownian motions. In fact the result still holds if $\mB$ and $\bar{\mB}$ are replaced by any field of continuous functions. We will use this fact in the application of the lemma. Also note that the middle terms only use the Brownian motions $B_0, \ldots, B_{n}$ (by Definition \ref{defn:p2l} of the point-to-point partition functions) while the outside terms use the Brownian motions $B_0, \ldots, B_n$ \text{and} $B_{n+1}$. Somewhat miraculously the contribution from the extra Brownian motion $B_{n+1}$ cancels off in the ratios, which is what allows the comparison to hold. Our proof of Lemma \ref{lm:comparison} is a modification of the proof of \cite[Lemma 3.8]{SV:ALEA2010}. 
\end{remark}

\begin{proof}
Throughout we write simply $\bar{Z}^{\lambda}_{(m,s),n}$ and $Z_{(m,s),(n,t)}$ and suppress the dependence on $\mB$ and $\bar{\mB}$. The proof is by induction. We start with \eqref{eqn:comparison2} and $n = 0$. In that case the middle term is
\[
\frac{Z_{(0,t),(0,T)}}{Z_{(0,s),(0,T)}} = e^{-B_0(s,t)}.
\]
Now the left inequality of \eqref{eqn:comparison2} follows from
\begin{align*}
\bar{Z}_{(0,t),0}^{\lambda}(\tau_0 < T) &= \int_t^T \exp \{-\lambda s_0 - B_1(s_0) + B_0(t,s_0) \} \, ds_0 \\
&\leq e^{-B_0(s,t)} \int_s^T \exp \{-\lambda s_0 - B_1(s_0) + B_0(s,s_0) \} \, ds_0 \\
&= e^{-B_0(s,t)} \bar{Z}_{(0,s),0}^{\lambda}(\tau_0 < T).
\end{align*}
For the right hand side use
\begin{align*}
\bar{Z}_{(0,t),0}^{\lambda}(\tau_0 > T) = \int_T^{\infty} \exp \{-\lambda s_0 - B_1(s_0) + B_0(t,s_0) \} \, ds_0 = e^{-B_0(s,t)} \bar{Z}_{(0,s),0}^{\lambda}(\tau_0 > T)
\end{align*}
so that in the $n = 0$ case the right hand side is actually an equality. Now for \eqref{eqn:comparison1} and $n = 1$ the middle term is
\begin{align*}
\frac{Z_{(0,0),(1,t)}}{Z_{(1,0),(1,t)}} 
&= e^{-B_1(t)} \int_0^t \exp \{ B_0(0, s_0) + B_1(s_0, t) \} \, ds_0\\ 
&= \int_0^t \exp \{ B_0(0,s_0) - B_1(0,s_0) \} \, ds_0.
\end{align*}
By definition we have
\[
\bar{Z}_{(1,0),1}^{\lambda}(\tau_1 < t) = \int_0^t e^{-\lambda s_1 - B_2(s_1) + B_1(0,s_1)} \, ds_1
\]
and then combining the last two equations gives
\begin{align*}
&\bar{Z}_{(1,0),1}^{\lambda}(\tau_1 < t) \frac{Z_{(0,0), (1,t)}}{Z_{(1,0),(1,t)}} 
= \int_0^t \int_0^t e^{-\lambda s_1 - B_2(s_1) + B_1(s_0, s_1) + B_0(0,s_0)} \, ds_0 \, ds_1 \\
&\qquad \geq \int \int e^{-\lambda s_1 - B_2(s_1) + B_1(s_0, s_1) + B_0(0, s_0) } \1{0 < s_0 < s_1 < t} \, ds_0 \, ds_1 \\
&\qquad= \bar{Z}_{(0,0),1}^{\lambda}(\tau_1 < t).
\end{align*}
This proves the left hand side of \eqref{eqn:comparison1} in the $n=1$ case. For the right hand side we have
\begin{align*}
\bar{Z}_{(1,0),1}^{\lambda}(\tau_1 > t) = \int_t^{\infty} e^{-\lambda s_1 - B_2(s_1) + B_1(0,s_1)} \, ds_1,   
\end{align*}
from which it follows that
\begin{align*}
\bar{Z}_{(1,0),1}^{\lambda}(\tau_1 > t) \frac{Z_{(0,0),(1,t)}}{Z_{(1,0),(1,t)}} &= \int_t^{\infty} \int_0^t e^{-\lambda s_1 - B_2(s_1) + B_1(s_0, s_1) + B_0(0,s_0)} \, ds_0 \, ds_1 \\ 
&\leq \int_t^{\infty} \int_0^{s_1} e^{-\lambda s_1 - B_2(s_1) + B_1(s_0, s_1) + B_0(0,s_0)} \, ds_0 \, ds_1 \\
&= \bar{Z}_{(0,0),1}^{\lambda}(\tau_1 > t).
\end{align*}
This proves the right hand side of \eqref{eqn:comparison1}, in the case $n=1$.

Now we complete the induction. Assume that \eqref{eqn:comparison1} holds for some $n+1$ and \eqref{eqn:comparison2} holds for $n$. Then we show \eqref{eqn:comparison2} holds for $n+1$ by using the decomposition
\[
\bar{Z}_{(0,s),n+1}^{\lambda}(\tau_{n+1} < T) = \bar{Z}_{(0,t),n+1}^{\lambda}(\tau_{n+1} < T) e^{B_0(s,t)} + \int_s^t \bar{Z}_{(1,u), n+1}^{\lambda}(\tau_{n+1} < T) e^{B_0(s,u)} \, du,
\]
which follows by decomposing the integral according to whether $t \leq s_0 < T$ or $s < s_0 < t$. Thus
\begin{align*}
&\frac{\bar{Z}_{(0,s),n+1}^{\lambda}(\tau_{n+1} < T)}{\bar{Z}_{(0,t),n+1}^{\lambda}(\tau_{n+1} < T)} \\
&\qquad= e^{B_0(s,t)} + \int_s^t \frac{\bar{Z}_{(1,u), n+1}^{\lambda}(\tau_{n+1} < T)}{\bar{Z}_{(1,t), n+1}^{\lambda}( \tau_{n+1} < T)} \frac{\bar{Z}_{(1,t), n+1}^{\lambda}(\tau_{n+1} < T)}{\bar{Z}_{(0,t), n+1}^{\lambda}(\tau_{n+1} < T)} e^{B_0(s,u)} \, du \\
&\qquad \geq e^{B_0(s,t)} + \int_s^t \frac{Z_{(1,u),(n+1,T)}}{Z_{(1,t),(n+1,T)}} \frac{Z_{(1,t),(n+1,T)}}{Z_{(0,t),(n+1,T)}} e^{B_0(s,u)} \, du \\
&\qquad= \frac{Z_{(0,s),(n+1,T)}}{Z_{(0,t),(n+1,T)}},
\end{align*}
with the last equality following from the analogous decomposition for $Z_{(0,s),(n+1,T)}$. This gives the left hand side of \eqref{eqn:comparison2} for $n+1$. Now we will show the left hand side of \eqref{eqn:comparison1} for $n+2$. Let $0 < s < t$, and then by decomposing the partition function $\bar{Z}_{(0,0),n+2}(\tau_{n+2} < t)$ according to whether $0 \leq s_0 \leq s$ or $s < s_0 \leq t$ we get
\begin{align*}
&\frac{\bar{Z}^{\lambda}_{(0,0),n+2}(\tau_{n+2} < t)}{\bar{Z}^{\lambda}_{(1,0),n+2}(\tau_{n+2} < t)} = \frac{\bar{Z}^{\lambda}_{(0,s),n+2}(\tau_{n+2} < t)}{\bar{Z}^{\lambda}_{(1,s),n+2}(\tau_{n+2} < t)} \frac{\bar{Z}^{\lambda}_{(1,s),n+2}(\tau_{n+2} < t)}{\bar{Z}^{\lambda}_{(1,0),n+2}(\tau_{n+2} < t)} e^{B_0(0,s)} \\ 
&\qquad\qquad\qquad\qquad\qquad\qquad + \int_s^t \frac{\bar{Z}^{\lambda}_{(1,u),n+2}(\tau_{n+2} < t)}{\bar{Z}^{\lambda}_{(1,0),n+2}(\tau_{n+2} < t)} e^{B_0(0,u)} \, du \\
&\qquad \leq \frac{\bar{Z}^{\lambda}_{(0,s),n+2}(\tau_{n+2} < t)}{\bar{Z}^{\lambda}_{(1,s),n+2}(\tau_{n+2} < t)} \frac{Z_{(1,s),(n+2,t)}}{Z_{(1,0),(n+2,t)}} e^{B_0(0,s)} + \int_s^t \frac{Z_{(1,u), (n+2,t)}}{Z_{(1,0),(n+2,t)}} e^{B_0(0,u)} \, du. 
\end{align*}
Note that in the last inequality we used the left hand side of \eqref{eqn:comparison2} for $n+1$, which we just proved. For shorthand let 
\[
a_{1,n+2}(s,t) = \frac{Z_{(1,s),(n+2,t)}}{Z_{(1,0),(n+2,t)}} e^{B_0(0,s)}. 
\]
Then by applying the same decomposition to $Z_{(0,0), (n+2,t)}$ we get
\[
\frac{Z_{(0,0),(n+2,t)}}{Z_{(1,0),(n+2,t)}} = \frac{Z_{(0,s),(n+2,t)}}{Z_{(1,s),(n+2,t)}} a_{1,n+2}(s,t) + \int_s^t a_{1,n+2}(u,t) \, du.
\]
Taking differences gives
\begin{align*}
&\frac{\bar{Z}^{\lambda}_{(0,0),n+2}(\tau_{n+2} < t)}{\bar{Z}^{\lambda}_{(1,0),n+2}(\tau_{n+2} < t)} - \frac{Z_{(0,0),(n+2,t)}}{Z_{(1,0),(n+2,t)}}\\
 &\qquad\qquad= \left( \frac{\bar{Z}^{\lambda}_{(0,s),n+2}(\tau_{n+2} < t)}{\bar{Z}^{\lambda}_{(1,s),n+2}(\tau_{n+2} < t)} - \frac{Z_{(0,s),(n+2,t)}}{Z_{(1,s),(n+2,t)}} \right) a_{1,n+2}(s,t).
\end{align*}
This holds for all $0 < s < t$ and the term on the left is independent of $s$. The right hand side goes to zero as $s \nearrow t$, which proves the left hand side of \eqref{eqn:comparison1} for the $n+2$ case. 

It now remains to prove the right hand inequalities of \eqref{eqn:comparison1} and \eqref{eqn:comparison2}, for $n+2$ and $n+1$, respectively. The arguments are straightforward modifications of those for the left hand inequalities, which we leave for the reader.
\end{proof}

\subsection{Proof of Theorem \ref{thm:almost_sure_check_comparison}}

First we prove a lemma that shows that the dominant contribution to the partition function $\bar{Z}_{(m,s),n}^{\lambda}$ comes from paths with asymptotic slope $\Psi_1(\lambda)$ (i.e. those for which $\tau_n \sim n \Psi_1(\lambda)$). This follows from Proposition \ref{prop:SLLN_refinement}, which is a refinement of the free energy result of \cite{MO:free_energy} for the O'Connell-Yor polymer.

\begin{lemma} \label{lem:shapeTheoremApplication}
Fix $\lambda>0$, $s>0$, and $m \in \Z$. If $0<\theta<\Psi_1(\lambda)$, then with probability one
\[
\lim_{n\to\infty} \frac{\bar{Z}^\lambda_{(m,s),n}(\mB, \bar{\mB}; \tau_n > n\theta)}{\bar{Z}^\lambda_{(m,s),n}(\mB, \bar{\mB})}=1.
\]
If $\theta>\Psi_1(\lambda)$, then with probability one
\[
\lim_{n\to\infty} \frac{\bar{Z}^\lambda_{(m,s),n}(\mB, \bar{\mB}; \tau_{n}< n\theta)}{\bar{Z}^\lambda_{(m,s),n}(\mB, \bar{\mB})}=1.
\]
\end{lemma}

\begin{proof}  For shorthand we write $\bar{Z}^{\lambda}_n = \bar{Z}^{\lambda}_{(m,s),n}$. Proposition \ref{prop:SLLN_refinement} gives the almost sure statements
\[
\lim_{n\to\infty}n^{-1} \log \bar{Z}^\lambda_{n} =
		\sup_{t>0}\{p(t)-\lambda t\}
 \]
  and
	\[ \lim_{n\to\infty} n^{-1}\log \bar{Z}^\lambda_{n} (\tau_n \le n \theta) = \sup_{0 < t \le \theta} \{p(t)-\lambda t\}.
\]
By Lemma \ref{lem:pconcave}, $p$ is a concave function. Thus $\sup_{t > 0} \{p(t) - \lambda t \}$  is a achieved at the unique point $t = \Psi_1(\lambda)$. This means the function $p(t)-\lambda t$ is strictly increasing for $t\in(0,\Psi_1(\lambda))$ and strictly decreasing for $t \in (\Psi_1(\lambda),\infty)$. Consequently, if $\theta < \Psi_1(\lambda)$, then $\sup_{t \le \theta} \{p(t) - \lambda t \}$ is strictly smaller than $p(\Psi_1(\lambda)) - \lambda \Psi_1(\lambda)$. This implies that for $\theta < \Psi_1(\lambda)$ there is a constant $c > 0$ such that, with probability one,
\[
\frac{\bar{Z}_n^{\lambda}(\tau_n \leq n \theta)}{\bar{Z}_n^{\lambda}} \leq e^{-c n + o(n)}
\]
for all but finitely many $n$. As a result
\[
\lim_{n \to \infty} \frac{\bar{Z}^\lambda_{n} (\tau_n \le n \theta)}{\bar{Z}^\lambda_{n}} = 0, 
\]
which proves the first statement of the lemma. The proof of the second statement is analogous, using the fact that
\[
\lim_{n\to\infty} n^{-1} \log \bar{Z}^\lambda_{n} (\tau_n \ge n \theta) = \sup_{t \ge \theta}\{p(t)-\lambda t\},
\]
which is strictly smaller than $p(\Psi_1(\lambda)) - \lambda \Psi_1(\lambda) = \sup_{t > 0} \{p(t) - \lambda t \}$ if $\theta > \Psi_1(\lambda)$.
\end{proof}

\begin{proof}[Proof of Theorem \ref{thm:almost_sure_check_comparison}] Fix $\lambda > 0$ and $0 < \delta < \Psi_1(\lambda)$. The comparison result of Lemma \ref{lm:comparison} applies to any pair of fields of Brownian motions, in particular to the fields $\check{\mB}^{\lambda}$ and $-g^{\lambda}$. For example, applying it to \eqref{eqn:comparison1} with $t = n \delta$ gives
\[
\frac{Z_{(0,0),(n,n\delta)}(\check{\mB}^{\lambda})}{Z_{(1,0),(n,n \delta)}(\check{\mB}^{\lambda})} \leq \frac{\bar{Z}^{\lambda}_{(0,0),n}(\check{\mB}^{\lambda}, -g^{\lambda}; \tau_n > n \delta)}{\bar{Z}^{\lambda}_{(1,0),n}(\check{\mB}^{\lambda}, -g^{\lambda}; \tau_n > n \delta)}.
\]
Now consider the limsup of both sides, and use Lemma \ref{lem:shapeTheoremApplication} to replace the right hand side with the unconstrained partition functions:
\begin{align*}
\limsup_{n \to \infty} \frac{Z_{(0,0),(n,n\delta)}(\check{\mB}^{\lambda})}{Z_{(1,0),(n,n\delta)}(\check{\mB}^{\lambda})} 
&\leq \limsup_{n \to \infty} \frac{\bar{Z}^{\lambda}_{(0,0),n}(\check{\mB}^{\lambda}, -g^{\lambda}; \tau_n > n \delta)}{\bar{Z}^{\lambda}_{(1,0),n}(\check{\mB}^{\lambda}, -g^{\lambda}; \tau_n > n \delta)} \\
&= \limsup_{n \to \infty} \frac{\bar{Z}^{\lambda}_{(0,0),n}(\check{\mB}^{\lambda}, -g^{\lambda})}{\bar{Z}^{\lambda}_{(1,0),n}(\check{\mB}^{\lambda}, -g^{\lambda})}.
\end{align*}
By Lemma \ref{lem:stationary_ratios}, the ratio on the right-hand side is independent of $n$ and equals $e^{r_1^\lambda(0)}$. This proves the left hand inequality in \eqref{eqn:almost_sure_check_comp_1} of the statement. The remaining inequalities in \eqref{eqn:almost_sure_check_comp_1} and \eqref{eqn:almost_sure_check_comp_2} are handled similarly.  
\end{proof}

\subsection{Proof of Theorem \ref{thm:ratios_of_part_funcs}: Convergence along Arbitrary Subsequences \label{sec:subsq}}

After the statement of Theorem \ref{thm:almost_sure_check_comparison} we showed how to use it to prove Theorem \ref{thm:ratios_of_part_funcs} in the case $t_n = n \theta$. Finally we show how to use the comparison result of Lemma \ref{lm:comparison} to extend Theorem \ref{thm:almost_sure_check_comparison} to arbitrary $t_n$ with  asymptotic speed $\theta$.

Let $t_n$ be any sequence such that $\lim_{n \to \infty} \frac{t_n}{n} = \theta$. For any $\epsilon > 0$, let $n$ be large enough so that $n(\theta - \epsilon) < t_n < n(\theta + \epsilon)$. This gives the containment of events $\{\tau_n > n(\theta + \epsilon) \} \subset \{\tau_n > t_n \} \subset \{ \tau_n > n( \theta - \epsilon) \}$. The comparison result Lemma \ref{lm:comparison} holds for any endpoint $t_n$ and any $\lambda > 0$, so
\begin{align*}
\frac{Z_{(0, 0), (n, t_n)}}{Z_{(1, 0), (n, t_n)}} \le \frac{\overline{Z}^\lambda_{(0, 0), n}(\tau_n > t_n)}{\overline{Z}^\lambda_{(1, 0), n}(\tau_n > t_n)} \le \frac{\overline{Z}^\lambda_{(0, 0), n}(\tau_n > n(\theta - \epsilon))}{\overline{Z}^\lambda_{(1, 0), n}(\tau_n > n(\theta + \epsilon))}.
\end{align*}
Now the first equation of Lemma \ref{lem:shapeTheoremApplication} can be applied for any $\lambda$ such that $\Psi_1(\lambda) > \theta + \epsilon$. Since $\Psi_1$ is monotonically decreasing, this condition is equivalent to $\lambda < \Psi_1^{-1}(\theta + \epsilon)$. As argued after the statement of Theorem \ref{thm:almost_sure_check_comparison}, by first taking the limsup as $n \to \infty$ and then a limit as $\lambda \nearrow \Psi_1^{-1}(\theta + \epsilon) = \lambda_{\theta}(\epsilon)$ we get the result
\[
\limsup_{n \to \infty} \frac{Z_{(0, 0), (n, t_n)}}{Z_{(1, 0), (n, t_n)}} \lesssim e^{r_1^{\lambda_{\theta}(\epsilon)}(0)}.
\]
As this is true for all $\epsilon > 0$, and $\Psi_1^{-1}(\theta + \epsilon) \nearrow \Psi_1^{-1}(\theta) = \lambda_{\theta}$, we can take the limit $\epsilon \searrow 0$ to get 
\[
\limsup_{n \to \infty} \frac{Z_{(0, 0), (n, t_n)}}{Z_{(1, 0), (n, t_n)}} \lesssim e^{r_1^{\lambda_{\theta}}(0)}.
\]
The remaining inequalities work in the same way, which completes the proof of equation \eqref{eqn:Busemann_exist_statement} of Theorem \ref{thm:ratios_of_part_funcs}.

\subsection{Proof of Theorem \ref{thm:ratios_of_part_funcs}: Tightness of the Path Measures \label{sec:tightness}}

To finish the proof that the infinite-length path measure exists, it remains to show that the family of measure $Q^{\mB}_{(0, 0), (n, t_n)}$ is almost surely relatively compact. By \cite[Lemmas 8.1 and 8.2]{EK:markov}, this amounts to showing the following. Intuitively, this result says that, as $n \to \infty$, there is no chance that jumps accumulate.

\begin{lemma}
Fix $T > 0$ and $\theta>0$. Then with probability one, for any sequence $t_n$ such that $t_n/n\to\theta$ we have
\[
\lim_{\delta \downarrow 0} \sup_n \sup_{k < n} Q^{\mB}_{(0, 0), (n, t_n)}(\tau_{k + 1} - \tau_k < \delta, \tau_k < T) = 0. 
\]
\end{lemma}

\begin{proof}
Fix $\lambda$ satisfying $0 < \psi_1(\lambda) < \theta$, which we will use later in order to apply Lemma \ref{lem:shapeTheoremApplication}. We will prove the result for the family of measures  $Q_{(0,0),(n,t_n)}^{\check{\mB}^{\lambda}}$ instead, which by the Burke property of the O'Connell-Yor polymer has the same law as the family $Q_{(0,0),(n,t_n)}^{\mB}$. We can use Lemma \ref{lem:quenched_marginals} to express the densities of $\tau_{k+1}$ and $\tau_k$ using ratios of partition functions. Letting $E_k(T)$ be the event $E_k(T)= \{ \tau_{k+1} - \tau_k < \delta, \tau_k < T \}$ we obtain
\begin{align}
Q^{\check{\mB}^{\lambda}}_{(0, 0), (n, t_n)}(E_k(T))
&= \int_0^T \int_{x}^{x + \delta} \frac{Z_{(0, 0), (k, x)} \cdot e^{\check{B}^{\lambda}_{k+1}(x, y)} \cdot Z_{(k + 2, y), (n, t_n)}}{Z_{(0, 0), (n, t_n)}} \, dy \, dx \notag \\
&\le \int_0^{T} \int_x^{x+\delta} Z_{(0, 0), (k, x)} e^{\check{B}^{\lambda}_{k+1}(x,y)} e^{-\check{B}^{\lambda}_{k+2}(0,y)} \frac{Z_{(k + 2, 0), (n, t_n)}}{Z_{(0,0),(n,t_n)}} \, dy  \, dx \notag \\
&\le \delta e^{A_k} \frac{Z_{(k+2,0),(n,t_n)}}{Z_{(0,0),(n,t_n)}} \int_0^T Z_{(0,0),(k,x)} \, dx, \label{eqn:tightness_bound}
\end{align}
where the $Z$ partition functions are all implicitly using the $\check{\mB}^{\lambda}$ weights, and $A_k$ is the stationary sequence of random variables
\[
A_k = \sup_{\substack{0 \leq x \leq y \leq T+\delta \\ |x-y| \leq \delta}} \check{B}^{\lambda}_{k+1}(x,y) -\inf_{0 \leq y \leq T+\delta} \check{B}^{\lambda}_{k+2}(0,y).
\]
Since $\E[e^{A_k}]$ is independent of $k$, and is finite by Gaussian tail bounds, Borel-Cantelli implies the event $\{ e^{A_k} \leq 2^k$ for all $k$ sufficiently large$\}$ has probability one. Similarly, since $\E[Z_{(0,0),(k,x)}] = e^{x/2} x^k/k!$ it follows that
\[
\E \left[ \int_0^T Z_{(0,0),(k,x)} \,dx \right] \leq T e^{T/2} \frac{T^k}{k!}, 
\]
and therefore by Borel-Cantelli the event $\{ \int_0^T Z_{(0,0), (k,x)} \, dx \leq (6 \lambda)^{-k}$ for all $k$ sufficiently large$\}$ has probability one. Applying this to \eqref{eqn:tightness_bound} gives that the event
\[
\Bigl\{\exists k_0\in\N:Q_{(0,0), (n,t_n)}^{\check{\mB}^{\lambda}}(E_k(T)) \leq \delta 2^k (6 \lambda)^{-k} \frac{Z_{(k+2,0), (n, t_n)}}{Z_{(0,0), (n,t_n)}} \textrm{ for all } n>k>k_0\Bigr\}
\]
has probability one. To complete the proof we will show that almost surely
\[
\frac{Z_{(k+2,0),(n,t_n)}}{Z_{(0,0),(n,t_n)}} < 2(2 \lambda)^k \textrm{ for all } k < n \textrm{ and } n \textrm{ sufficiently large},
\]
for our choice of $\lambda > 0$. Then the above two bounds imply that
\[
\sup_n \sup_{k < n} Q_{(0,0), (n,t_n)}^{\check{\mB}^{\lambda}}(E_k(T)) \leq C \delta
\]
for some $C = C(\omega) < \infty$ (almost surely), from which the result follows. 

For the last part use \eqref{eqn:comparison1} from the comparison lemma to obtain the bound
\begin{align*}
\frac{Z_{(k+2,0),(n,t_n)}}{Z_{(0,0),(n,t_n)}} 
&\leq \frac{\bar{Z}^{\lambda}_{(k+2,0),n}(\check{\mB}^{\lambda}, -g^{\lambda}; \tau_n < t_n)}{\bar{Z}^{\lambda}_{(0,0),n}(\check{\mB}^{\lambda}, -g^{\lambda}; \tau_n < t_n)} \\
&\leq \frac{\bar{Z}^{\lambda}_{(k+2,0),n}(\check{\mB}^{\lambda}, -g^{\lambda})}{\bar{Z}^{\lambda}_{(0,0),n}(\check{\mB}^{\lambda}, -g^{\lambda})}\cdot \frac{\bar{Z}^{\lambda}_{(0,0),n}(\check{\mB}^{\lambda}, -g^{\lambda})}{\bar{Z}^{\lambda}_{(0,0),n}(\check{\mB}^{\lambda}, -g^{\lambda}; \tau_n < t_n)}. 
\end{align*}
The above holds for all $\lambda > 0$ since it only relies on the comparison lemma. Now we use that $\lambda$ satisfies $\Psi_1(\lambda) < \theta$, so that by Lemma \ref{lem:shapeTheoremApplication} the second ratio on the right hand side converges to one almost surely, and therefore the event
\[
\Bigl\{\frac{Z_{(k+2,0), (n, t_n)}}{Z_{(0,0), (n,t_n)}} \leq 2 \frac{\bar{Z}^{\lambda}_{(k+2,0),n}(\check{\mB}^{\lambda}, -g^{\lambda})}{\bar{Z}^{\lambda}_{(0,0),n}(\check{\mB}^{\lambda}, -g^{\lambda})} \textrm{ for all } k < n \textrm{ and } n \textrm{ sufficiently large }\Bigr\}
\]
has probability one. Finally, use the Burke property of Lemma \ref{lem:stationary_ratios} to obtain
\[
 \frac{\bar{Z}^{\lambda}_{(k+2,0),n}(\check{\mB}^{\lambda}, -g^{\lambda})}{\bar{Z}^{\lambda}_{(0,0),n}(\check{\mB}^{\lambda}, -g^{\lambda})} =  \exp \left \{ - \sum_{i=1}^{k+2} r_i^{\lambda}(0) \right \}.
\]
 Note that the right hand side is independent of $n$. Since the $e^{-r_i^{\lambda}(0)}$ are iid $\operatorname{Gamma}(\lambda, 1)$ random variables, which have mean $\lambda$, it follows by Borel-Cantelli that the event 
\[
\Bigl\{\frac{\bar{Z}^{\lambda}_{(k+2,0),n}(\check{\mB}^{\lambda}, -g^{\lambda})}{\bar{Z}^{\lambda}_{(0,0),n}(\check{\mB}^{\lambda}, -g^{\lambda})} \leq (2\lambda)^k \textrm{ for all } k < n \textrm{ and } n \textrm{ sufficiently large}\Bigr\}
\]
has probability one. 
\end{proof}

\section{Limits of Brownian LPP \label{sec:LPP}}
In this section, we prove the limit result for the Brownian last passage percolation model. As before, let $\mB = (B_i(t), t \in \R)_{i \in \Z}$ be a field of independent two-sided Brownian motions. For space-time points $(m, s) \le (n, t)$ and a sequence of jump times $s = s_{m-1} < s_m < \dots < s_n = t$, we let $\sum_{k = m}^n B_k(s_{k -1}, s_k)$ be the ``length'' of the c\`adl\`ag path defined by the jump times in the random environment. Under the polymer measure sets with longer length have larger probability, while the last passage model picks the longest path. 

\begin{definition}\label{defn:LPP} 
Let $(m, s), (n, t) \in \Z \times \R$ with $(m, s) < (n, t)$. The \emph{last-passage time} is defined as 
\[
L_{(m, s), (n, t)}(\mB) = \sup \left\lbrace \sum_{k = m}^n B_k(s_{k-1}, s_k ) : s = s_{m-1} < s_m < \dots < s_n = t \right\rbrace.
\]
For $m=n$ and $s \leq t$ we set $L_{(m,s), (m,t)} = B_m(s,t)$.
\end{definition}

With this definition we get the following result for limits of differences of passage times:

\begin{theorem} \label{thm: LPPresult}
Fix $\mx, \my \in \Z \times \R$ and $\theta \in \R$. Then with probability one, the limit
\[
\lim_{n \to \infty} (L_{\mx, (n, t_n)} - L_{\my, (n, t_n)}) = \mathcal{B}^{\theta}_{\infty}(\mx, \my)
\]
exists almost surely and is independent of the choice of the sequence $\{ t_n \}$, so long as $\lim_{n \to \infty} \frac{t_n}{n} = \theta$.
\end{theorem}

\begin{remark}
Clearly the last passage time can be realized as the zero temperature limit of the polymer free energy 
\[
\lim_{\beta \to \infty} \frac{1}{\beta} \log Z_{(m,s),(n,t)}(\beta \mB) = L_{(m,s),(n,t)}(\mB).
\]
Ideas from \cite{BL:thermo_limit, BL:zero_temp_limit, RJ:DLR_Busemann} can be used to show that weak convergence holds at the level of Busemann functions, namely that
\[
\frac{1}{\beta} \mathcal{B}^{\theta}(\mx, \my; \beta \mB) \mathop{\Longrightarrow}_{\beta \to \infty} \mathcal{B}^{\theta}_{\infty}(\mx, \my,\mB).
\]
\end{remark}

Since the proofs are analogous to the polymer case, we simply state the necessary lemmas and the main ideas, leaving the full details to the reader.

\subsection{The Stationary Model}

As in the polymer model, the strategy is to define new fields of Brownian motion from the original $\mB$ which satisfy a Burke property. We can define a version of a point-to-line passage time using the new weights, and due to the recursive construction of the weights, the differences of this passage time are independent of the weight on the terminal line.

\begin{definition}
Fix $\lambda > 0$. For $N \in \N$ and $T \in \R$, set
\begin{align}\label{eqn:LPP_field_defns}
\begin{split}
&q_N^\lambda(T) = \sup_{- \infty < s \le T} \left\lbrace B_N(s, T) + f^\lambda_{N - 1}(s, T) - \lambda (T - s) \right\rbrace \\
&f^\lambda_N(T) = f_{N-1}^\lambda(T) + q_N^\lambda(0) - q_N^\lambda(T). 
\end{split}
\end{align}
As initial conditions for the recursion, set $f_0^\lambda(T) = B_0(T)$. For $N \ge 1$, define a sequence of field $\tilde{B}^\lambda$ by
\[
\tilde{B}_{N-1}^\lambda(S, T) = B_N(S, T) - q_N^\lambda(S, T)
\]
\end{definition}

By \cite[Theorem 2]{OCY:SPA2001}, these fields satisfy a Burke property and are independent similarly to Theorem \ref{thm:burke}. Furthermore, the random variables $q_i^{\lambda}(0)$ have an exponential distribution with mean $1/\lambda$. We now define a point-to-line passage time which uses special weights on the boundary line:
\[
\overline{L}_{(m, s), n}^\lambda(\mB, \overline{\mB}) = \sup \Bigl\{ \sum_{k = m}^n B_k(s_{k-1}, s_k ) - \overline{B}_{n+1}(s_n) - \lambda s_n : s = s_{m-1} < s_m < \dots < s_n  \Bigr\}.
\]
When the weights $\tilde{\mB}^\lambda$ and $-f^\lambda$ are used, the differences in the lengths become independent of $n$, and only depend on differences in the starting points. 

\begin{lemma}
For $N \in \N$, the following identities hold:
\begin{align*}
&\overline{L}_{(0, t), N-1}^{\lambda}(\tilde{\mB}^\lambda, - f^\lambda) - \overline{L}_{(0,s), N-1}^{\lambda}(\tilde{\mB}^\lambda, - f^\lambda) = B_0(s, t) - \lambda(t - s), \,\,\,\,\,\, \text{and} \\
&\overline{L}_{(0, t), N}^{\lambda}(\tilde{\mB}^\lambda, - f^\lambda) - \overline{L}_{(1, t), N}^{\lambda}(\tilde{\mB}^\lambda, - f^\lambda) = q_1^\lambda(t).
\end{align*}
\end{lemma}

\begin{proof}
The proof is inductive in the same way as the proof of Lemma \ref{lem:stationary_ratios}. The key are the recursions \eqref{eqn:LPP_field_defns} and the identity
\begin{equation} \label{eqn: qIdentity}
q_n^\lambda(T) = \sup_{T \le s < \infty} \left\lbrace \tilde{B}^\lambda_{n-1}(T, s)  + f^\lambda_n(T, s) + \lambda (T - s) \right\rbrace,
\end{equation} 
which follows from \cite[Theorem 3]{OCY:SPA2001}. Thus for $n > m$,
\begin{align*}
&\overline{L}^\lambda_{(m, s), n} (\tilde{\mB}^\lambda, - f^\lambda) = \sup_{s_n > s} \left\lbrace L_{(m, s), (n, s_n)} (\tilde{\mB}^\lambda) + f^\lambda_{n+1}(s_n) - \lambda s_n \right\rbrace \\
&\qquad= \sup_{s_n > s} \left\lbrace \sup_{s < s_{n-1} < s_n} \left( L_{(m, s), (n-1, s_{n-1})}(\tilde{\mB}^\lambda) + \tilde{B}_n(s_{n-1}, s_n) \right) + f^\lambda_{n+1}(s_n) - \lambda s_n \right\rbrace \\
&\qquad= \sup_{s_{n-1} > s} \left\lbrace L_{(m, s), (n-1, s_{n-1})}(\tilde{\mB}^\lambda) + f_{n+1}^\lambda(s_{n-1}) - \lambda s_{n-1} + q_{n+1}^\lambda(s_{n-1}) \right\rbrace \\
&\qquad= \sup_{s_{n-1} > s} \left\lbrace L_{(m, s), (n-1, s_{n-1})}(\tilde{\mB}^\lambda) + f_{n}^\lambda(s_{n-1}) - \lambda s_{n-1} + q^\lambda_{n+1}(0) \right\rbrace \\
&\qquad= q_{n+1}^\lambda(0) + \overline{L}^\lambda_{(m, s), n-1} (\tilde{\mB}^\lambda, - f^\lambda).
\end{align*}
Note that we used \eqref{eqn: qIdentity} to get the $q_{n+1}^\lambda(s_{n-1})$ term in the third line. Iterating this identity gives
\begin{align*}
\overline{L}^\lambda_{(m, s), n}(\tilde{\mB}^\lambda, -f^\lambda) 
= f_m^\lambda(s) - \lambda s + \sum_{k = m + 1}^{n+1} q_k^\lambda(0). 
\end{align*}
From this and the construction of the $f^\lambda, q^\lambda$ fields, the lemma immediately follows.
\end{proof}

\subsection{Comparison Lemma}

As before, we restrict the point-to-line length $\overline{L}^\lambda$ to the event $\{\tau_n < t \}$ or $\{\tau_n > t \}$. For example,
\[
\overline{L}^\lambda_{(m, s), n}(\mB, \overline{\mB}; \tau_n\! <\! t) =  \sup \Bigl\{\sum_{k = m}^n\! B_k(s_{k-1}, s_k ) - \overline{B}_{n+1}(s_n) - \lambda s_n \!: \mathbf{s}_{m,n} \!\in \Pi_{(m, s), n}, s_n < t \Bigr\}.
\]
To shorten notation, write $J_{1, (n, t)}(\mB) = L_{(0, 0), (n, t)}(\mB) - L_{(1, 0), (n, t)}(\mB)$ for the difference in lengths starting from neighboring levels. For $0 < s < t < T$, write 
\[
I_{(s, t), (n, T)}(\mB) = L_{(0, t), (n, T)}(\mB) - L_{(0 ,s), (n, T)}(\mB).
\]
Also write $\overline{J}_{1, n}^\lambda = \overline{L}^{\lambda}_{(0,0),n} - \overline{L}^{\lambda}_{(1,0),n}$ and $\overline{I}^\lambda_{(s,t), n} = \overline{L}^{\lambda}_{(0,t),n} - \overline{L}^{\lambda}_{(0,s),n}$. In this case the comparison lemma follows from a now standard paths crossing argument. See also \cite{CP:LPP_Busemann} for a related argument.

\begin{lemma}
Fix $\lambda > 0$. Let $t > 0$ and $n \in\N$. Then
\begin{equation} \label{eqn: lengthCL1}
\overline{J}^\lambda_{1, n}(\mB, \overline{\mB}; \tau_n < t) \le J_{1, (n, t)}(\mB) \le \overline{J}^\lambda_{1, n}(\mB, \overline{\mB}; \tau_n > t).
\end{equation}
Similarly, let $0 < s < t < T$ and $n \in\Z_+$. Then
\begin{equation} \label{eqn: lengthCL2}
\overline{I}^\lambda_{(s, t), n}(\mB, \overline{\mB}; \tau_n < T) \le I_{(s, t), (n, T)}(\mB) \le \overline{I}^\lambda_{(s, t), n}(\mB, \overline{\mB}; \tau_n > T).
\end{equation}
\end{lemma}

\begin{proof}
We omit $\mB, \overline{\mB}$, as it is clear from the context which lengths use which fields. The key to the proof of the upper bound of \eqref{eqn: lengthCL1} is to observe that the two paths which achieve $L_{(0, 0), (n, t)}$ and $\overline{L}^\lambda_{(1, 0), n}(\tau_n > t)$ must cross at some space-time point. For a fixed realization of Brownian motions $\mB$ and $\overline{\mB}$, suppose that the first point that the two paths meet is $\my$, where $(1, 0) \le \my \le (n, t)$. This means that the lengths can be decomposed as
\begin{align*}
L_{(0, 0), (n, t)} = L_{(0, 0), \my} + L_{\my, (n, t)}, \quad \overline{L}^\lambda_{(1, 0), n}(\tau_n > t) = L_{(1, 0), \my} + \overline{L}^\lambda_{\my, n}(\tau_n > t).
\end{align*}
Since the path that achieves $L_{(1, 0), (n, t)}$ does \emph{not} necessarily pass through $\my$, we can only say $L_{(1, 0), (n, t)} \ge L_{(1, 0), \my} + L_{\my, (n, t)}$.  Thus,
\begin{align*}
J_{1, (n,t)} &= L_{(0, 0), (n, t)} - L_{(1, 0), (n, t)} \\
&\le \left( L_{(0, 0), \my} + L_{\my, (n, t)}\right) - \left( L_{(1, 0), \my} + L_{\my, (n, t)} \right)  = L_{(0, 0), \my} - L_{(1,0), \my} \\
& = \left( L_{(0, 0), \my} + \overline{L}^\lambda_{\my, n}(\tau_n > t)\right) - \left( L_{(1, 0), \my} + \overline{L}^\lambda_{\my, n}(\tau_n > t) \right) \\
& \le  \overline{L}^\lambda_{(0, 0), n}(\tau_n > t) -  \overline{L}^\lambda_{(1, 0), n}(\tau_n > t) = \overline{J}^\lambda_{1, n}(\tau_n > t),
\end{align*}
where the final inequality again comes from the observation that $\overline{L}^\lambda_{(0, 0), n}(\tau_n > t) \le L_{(0, 0), \mx} + \overline{L}^\lambda_{\mx, n}(\tau_n > t)$, for any $\mx \ge (0, 0)$.

For the lower bound in \eqref{eqn: lengthCL1}, use that the paths achieving $L_{(1, 0), (n, t)}$ and $\overline{L}^\lambda_{(0, 0), n}(\tau_n < t)$ must cross at some point, say $\mz$. Then,
\begin{align*}
J_{1, (n,t)} &= L_{(0, 0), (n, t)} - L_{(1, 0), (n, t)} \\
&\ge \left( L_{(0, 0), \mz} + L_{\mz, (n, t)}\right) - \left( L_{(1, 0), \mz} + L_{\mz, (n, t)} \right)  = L_{(0, 0), \mz} - L_{(1,0), \mz} \\
& = \left( L_{(0, 0), \mz} + \overline{L}^\lambda_{\mz, n}(\tau_n < t)\right) - \left( L_{(1, 0), \mz} + \overline{L}^\lambda_{\mz, n}(\tau_n < t) \right) \\
& \ge  \overline{L}^\lambda_{(0, 0), n}(\tau_n < t) -  \overline{L}^\lambda_{(1, 0), n}(\tau_n < t) = \overline{J}^\lambda_{1, n}(\tau_n < t).
\end{align*}
The bounds in \eqref{eqn: lengthCL2} are similar, and all that is needed is to determine which two of the four paths involved are forced to cross. 
\end{proof}

\subsection{Shape Theorem for the Passage Time}

The final ingredient necessary to prove Theorem \ref{thm: LPPresult} is the analogue of Lemma \ref{lem:shapeTheoremApplication}, which enables us to remove the restriction to the events $\{\tau_n < nt\}$ or $ \{ \tau_n > nt \}$ in the limit as $n \to \infty$. As in the polymer case, this can be done using the shape theorem
\begin{equation}
\lim_{n \to \infty} n^{-1} L_{(0, 0), (n, nt)} (\mB) = 2 \sqrt{t},
\end{equation}
almost surely, which was proven in \cite{HMO:concentration}. Analogously to Proposition \ref{prop:SLLN_refinement} one can translate this into an almost sure limit statement about the point-to-line passage time. Indeed, fix $s \leq S < T \leq \infty$. Then, almost surely
\[
\lim_{n \to \infty} n^{-1}  \overline{L}_{(m,s),n}^{\lambda}(nS \leq \tau_n \leq nT) = \sup_{S \leq t \leq T} \left \{ 2 \sqrt{t} - \lambda t \right \}.
\]
Note that $2 \sqrt{t} - \lambda t$ is a concave function of $t$ with a unique maximum at $t_\lambda := 1/\lambda^2$. The analogue of Lemma \ref{lem:shapeTheoremApplication} is the following.

\begin{lemma} 
Fix $\lambda>0$, $s>0$, and $m\in\Z$.  If $0<\theta<\lambda^{-2}$, then with probability one
\[
\lim_{n\to\infty} \frac{\bar{L}^\lambda_{(m,s),n}(\mB, \bar{\mB}; \tau_n > n\theta)}{\bar{L}^\lambda_{(m,s),n}(\mB, \bar{\mB})}=1.
\]
If $\theta > \lambda^{-2}$, then with probability one
\[
\lim_{n\to\infty} \frac{\bar{L}^\lambda_{(m,s),n}(\mB, \bar{\mB}; \tau_{n}< n\theta)}{\bar{L}^\lambda_{(m,s),n}(\mB, \bar{\mB})}=1.
\]
\end{lemma}

\appendix
\section{Shape Theorem for the Free Energy}

In \cite{MO:free_energy} the authors prove that the free energy of the point-to-point O'Connell-Yor polymer is
\begin{align}\label{eqn:MO_free_energy}
\lim_{n \to \infty} n^{-1} \log Z_{(0,0), (n,nt)} = \inf_{\lambda > 0} \{ \lambda t - \Psi_0(\lambda) \} =: p(t),
\end{align}
where $\Psi_0 = \Gamma'/\Gamma$ is the digamma function. Note that this result holds for each $t > 0$ and that the limit is in the almost sure sense. The infimum is uniquely achieved at $\lambda_* = \Psi_1^{-1}(t)$, where $\Psi_1 = \Psi_0'$ is the trigamma function. The next theorem extends this result by showing that the asymptotic free energy behaves locally as predicted by convex duality, when the paths have free endpoints but are restricted to go in certain asymptotic directions. We prove the result for two versions of the free endpoint partition function. First recall the partition function $\bar{Z}_{(m,s),n}^{\lambda}(\mB, \bar{\mB})$ from Definition \ref{defn:Z_bar}, which for the rest of this section we shorten to simply $\bar{Z}_{(m,s),n}^{\lambda}$. For $s \leq S < T \leq \infty$ we also let
\begin{align}
\bar{Z}_{(m,s),n}^{\lambda}(S \leq \tau_n \leq T) &= \int_{\Pi_{(m,s),n}} \!\!\!\!\! e^{ \sum_{k=m}^n B_k(s_{k-1}, s_k) - \bar{B}_{n+1}(s_n) - \lambda s_n } \1{S \leq s_n \leq T} \, d \ms_{m,n} \notag \\
&= \int_S^T e^{-\bar{B}_{n+1}(x) - \lambda x} Z_{(m,s), (n,x)}(\mB) \, dx. \label{eqn:barZ_endpoint_in_interval}
\end{align}

\begin{proposition}\label{prop:SLLN_refinement}
Fix $s \leq S < T \leq \infty$. Then almost surely
\[
\lim_{n \to \infty} n^{-1} \log \overline{Z}_{(m,s),n}^{\lambda}(nS \leq \tau_n \leq nT) = \sup_{S \leq t \leq T} \left \{ p(t) - \lambda t \right \}.
\]
\end{proposition}

As the proof will make clear, the extra effect of the term $\bar{B}_{n+1}$ is negligible. Thus if we replaced $\bar{Z}_{(m,s),n}^{\lambda}$ with the partition function
\[
Z_{(m,s),n}^{\lambda}(\mB; S \leq \tau_n \leq T) = \int_{\Pi_{(m,s),n}} \!\!\!\!\! \exp \left \{ \sum_{k=m}^n B_k(s_{k-1}, s_k) - \lambda s_n \right \} \1{S \leq s_n \leq T} \, d \ms_{m,n}
\]
the same result will hold.

To prove Proposition \ref{prop:SLLN_refinement} we will need the following simple fact, which was already used in Lemma \ref{lem:shapeTheoremApplication}.

\begin{lemma}\label{lem:pconcave}
The function $p$ is strictly concave.
\end{lemma}

\begin{proof}
Define a function $f$ by $f(\lambda) = -\Psi_0(-\lambda)$ for $\lambda < 0$, and $f(\lambda) = \infty$ for $\lambda \geq 0$. Then by definition $p$ is the convex dual of $f$. Since $f$ is differentiable it follows that $-p$ is strictly convex, and therefore $p$ is strictly concave.
\end{proof}

\begin{proof}[Proof of Proposition \ref{prop:SLLN_refinement}]
For shorthand we write $\bar{Z}_n^{\lambda} = \bar{Z}_{(m,s),n}^{\lambda}$. We will first prove the lower bound by showing that
\begin{align}\label{eqn:SLLN_refinement_lower_bound}
\liminf_{n \to \infty} n^{-1} \log \overline{Z}_n^{\lambda}(nS \leq \tau_n \leq nT) \geq p(t) - \lambda t
\end{align}
for all $t \in [S,T)$. In the case $T < \infty$ the extension to all $t \in [S,T]$ then follows by continuity. To prove the above fix $\epsilon > 0$ and $t \in [S,T)$. Then for sufficiently large $n$ we have $\epsilon < n(T-t)$, or equivalently $nt + \epsilon < nT$. From this and equation \eqref{eqn:barZ_endpoint_in_interval} defining the partition function we obtain
\begin{align*}
\overline{Z}_n^{\lambda}(nS \leq \tau_n \leq nT) &\geq \int_{nt}^{nt + \epsilon} e^{-\bar{B}_{n+1}(x) - \lambda x} Z_{(0,0),(n,x)} \, dx \\
& = Z_{(0,0), (n,nt)} \int_{nt}^{nt + \epsilon} e^{-\bar{B}_{n+1}(x) - \lambda x} \frac{Z_{(0,0),(n,x)}}{Z_{(0,0), (n,nt)}} \, dx \\
& \geq Z_{(0,0), (n,nt)} \int_{nt}^{nt + \epsilon} e^{-\bar{B}_{n+1}(x) - \lambda x} Z_{(n,nt),(n,x)} \, dx,
\end{align*}
with the last inequality following from the supermultiplicativity property \eqref{supermult} of the partition functions. Now use that
\[
Z_{(n,nt),(n,x)} = e^{B_n(nt,x)}
\]
and lower bound the last integral by its infimum value over the length of the interval. Then take logarithms to obtain
\begin{align*}
\log \overline{Z}^{\lambda}_n(nS \leq \tau_n \leq nT) 
\geq \log \epsilon &+ \log Z_{(0,0), (n,nt)} - \lambda(nt + \epsilon)\\
&+ \!\! \inf_{nt \leq x \leq nt + \epsilon} \!\!\!\! B_n(nt, x) - \sup_{nt \leq x \leq nt + \epsilon} \!\!\!\! \overline{B}_{n+1}(x).
\end{align*}
Now divide both sides by $n$. Standard Gaussian tail bounds and the scaling property of Brownian motion imply that the last two terms will go to zero as $n \to \infty$, and by \eqref{eqn:MO_free_energy} the second term on the right converges to $p(t)$. This proves \eqref{eqn:SLLN_refinement_lower_bound}.

The proof of the upper bound is a tighter version of the above. We first consider the case $T < \infty$ and then later extend to the case $T = \infty$. Fix $k \in \N$ and subdivide the interval $[nS,nT]$ into $k$ equally spaced pieces by letting $x_{n,i} = nS + ni(T-S)/k$ for $0 \leq i \leq k$. First use that
\begin{align*}
\overline{Z}_n^{\lambda}(nS \leq \tau_n \leq nT) &\leq k \max_{1 \leq i \leq k} \int_{x_{n,i-1}}^{x_{n,i}} e^{-\overline{B}_{n+1}(x) - \lambda x} Z_{(0,0),(n,x)} \, dx \\
&\leq k \max_{1 \leq i \leq k} e^{-\lambda x_{n,i-1}} Z_{(0,0),(n, x_{n,i})} \int_{x_{n,i-1}}^{x_{n,i}} \frac{e^{-\overline{B}_{n+1}(x)}}{Z_{(n,x), (n,x_{n,i})}} \, dx,
\end{align*}
with the last inequality again following by the supermultiplicativity property \eqref{supermult}. Now by replacing the integrands with their maximal values over the respective intervals and taking logarithms we obtain
\begin{align*}
\log \overline{Z}_n^{\lambda}(nS \leq \tau_n \leq nT) \leq \log &\, k + \log (n(T-S)/k) + \max_{1 \leq i \leq k} \{ \log Z_{(0,0), (n, x_{n,i})} - \lambda x_{n, i-1} \} \\
&+ \sup_{nS \leq x \leq nT} |\overline{B}_{n+1}(x)| + \max_{1 \leq i \leq k} \sup_{nS \leq x \leq nT} |B_n(x,x_{n,i})| 
\end{align*}
Upon dividing by $n$ the first two terms on the right go to zero for obvious reasons, while the last two go to zero by the same Gaussian tail bounds and scaling properties of Brownian motion as before. Therefore, using \eqref{eqn:MO_free_energy} on the remaining term on the right hand side and the definition of $x_{n,i}$ we have
\begin{align*}
&\limsup_{n \to \infty} n^{-1} \log \overline{Z}_n^{\lambda}(nS \leq \tau_n \leq nT) \\
&\qquad\qquad\leq \max_{1 \leq i \leq k} \{ p(S + i(T-S)/k) - \lambda(S + (i-1)(T-S)/k) \} \\ 
&\qquad\qquad\leq \sup_{S \leq t \leq T} \{ p(t) - \lambda t \} + \lambda (T-S)/k.
\end{align*}
Taking $k \to \infty$ completes the proof in the case $T < \infty$. 

For the case $T = \infty$, first observe that for any $\lambda > 0$ there is a $T$ sufficiently large such that
\[
\sup_{S \leq t \leq T} \{ p(t) - \lambda t \} = \sup_{S \leq t} \{ p(t) - \lambda t \}.
\]
This follows from the fact that $p$ is strictly concave, and thus $p(t) - \lambda t$ has a unique maximum. Now since the statement of the proposition holds for $T < \infty$ and
\[
\overline{Z}_n^{\lambda}(nS \leq \tau_n) = \overline{Z}_n^{\lambda}(nS \leq \tau_n \leq nT) + \overline{Z}_n^{\lambda}(\tau_n > nT), 
\]
it is enough to show that
\[
\lim_{T \to \infty} \limsup_{n \to \infty} n^{-1} \log \overline{Z}_n^{\lambda}(\tau_n > nT) = -\infty.
\]
To this end, first use the bound
\begin{equation} \label{eqn: greaterNT}
\begin{split}
\overline{Z}_n^{\lambda}(\tau_n > nT) 
&= \int_{nT}^{\infty} e^{-\overline{B}_{n+1}(x) - \lambda x} Z_{(0,0), (n,x)} \, dx\\
& \leq e^{-nT/2} \int_{0}^{\infty} e^{-\overline{B}_{n+1}(x) - \lambda x/2} Z_{(0,0), (n,x)} \, dx. 
\end{split}
\end{equation}
It will be enough to show that the integral term grows at most exponentially fast. For this, we can use the bound 
\[
Z_{(0,0),(n,x)} = \int_{\Pi_{(0,0),(n,x)}} \!\!\!\!\!\! \exp \left \{ \sum_{j=0}^{n} B_j(s_{j-1}, s_j) \right \} \, d \ms_{0,n} \leq \int_{\Pi_{(0,0),(n,x)}} \!\!\!\!\!\! e^{L_n(x)} \, d \ms_{0,n} = \frac{x^n}{n!}\cdot e^{L_n(x)},
\] 
where $L_n(x)$ is the maximal energy of a path from $(0,0)$ to $(n,x)$, i.e.
\[
L_n(x) = \sup \Biggl \{ \sum_{j=0}^n B_j(s_{j-1}, s_j) : 0 = s_{-1} < s_0 < \ldots < s_{n-1} < s_n = x \Biggr \}.
\]
Using this bound, Stirling's approximation, and the substitution $y = x/n$, the integral term is:
\begin{align*}
\int_{0}^{\infty} e^{-\overline{B}_{n+1}(x) - \lambda x/2} Z_{(0,0), (n,x)} \, dx & \le \frac{1}{n!} \int_0^\infty x^n e^{- \overline{B}_{n+1}(x) - \lambda x/2} \cdot e^{L_n(x)} \, dx \\
& \le \frac{e^n}{\sqrt{2 \pi n}} \int_0^\infty \left( \frac{x}{n} \right)^n e^{- \overline{B}_{n+1}(x) - \lambda x/2} \cdot e^{L_n(x)} \, dx \\
& \le C e^{2n} \int_0^\infty y^n e^{- \overline{B}_{n+1}(ny) - \lambda ny/2} \cdot e^{L_n(ny)} \, dy.
\end{align*}
We can now argue that $L_n(ny)$ and $- \overline{B}_{n+1}(ny)$ almost surely grow at most linearly in $n$. Theorem 2 of \cite{HMO:concentration} proves that with probability one
\[
\lim_{n \to \infty} \sup_{y > 0} \left| \frac{\frac{1}{n} L_n(ny) - 2 \sqrt{y}}{1+y} \right| = 0. 
\]
Therefore, on a set of full probability there exists an $N_1 = N_1(\omega) < \infty$ such that
\[
L_n(ny) \leq \frac{\lambda}{8}n(1+y) + 2 n \sqrt{y} 
\]
for all $n \geq N_1$ and all $y > 0$. Similarly, it is straightforward to show that there exists a set of full probability on which
\[
-\overline{B}_{n+1}(ny) \leq \frac{\lambda}{8}n(1 + y)
\]
for all $n \geq N_2 = N_2(\omega)$ (with $N_2(\omega) < \infty$) and all $y > 0$. Thus, on a set of full probability 
\begin{align*}
\int_0^{\infty} e^{-\overline{B}_{n+1}(x) - \lambda x/2} Z_{(0,0),(n,x)} \, dx &\leq C e^{2n} \int_0^{\infty} y^n e^{\lambda n (1+y)/4 + 2n \sqrt{y} -\lambda ny/2} \, dy \\
&= C e^{(2 + \lambda/4)n} \int_0^{\infty} y^n e^{-\lambda n y/4 + 2 n\sqrt{y}} \, dy
\end{align*}
for all $n \geq \max \{N_1,  N_2\}$. Thus, there is a constant $K$ (depending on $\lambda$) such that
\[
\limsup_{n \to \infty} n^{-1} \log \left( \int_{0}^{\infty} e^{-\overline{B}_{n+1}(x) - \lambda x/2} Z_{(0,0), (n,x)} \, dx \right) \le K,
\]
almost surely. Using \eqref{eqn: greaterNT}, this gives the final result
\[\lim_{T \to \infty} \limsup_{n \to \infty} n^{-1} \log \overline{Z}_n^\lambda( \tau_n > nT) \le \lim_{T \to \infty}  \left( \frac{-T}{2} + K \right)  = - \infty. \qedhere
\]
\end{proof}

\section*{Acknowledgements}
We thank Chris Janjigian for valuable discussions and several helpful comments on the structure of the paper, and Evan Sorensen for helpful comments and for pointing out typos in an earlier draft. Tom Alberts is partially supported by National Science Foundation grants DMS-1811087 and DMS-1715680. Mackenzie Simper is supported by a National Defense Science \& Engineering Graduate Fellowship. Firas Rassoul-Agha was partially supported by National Science Foundation grants DMS-1407574 and DMS-1811090.

\bibliographystyle{imsart-number}
\bibliography{citations}

\end{document}